\documentclass[10pt]{article}
\usepackage{mathrsfs,color}
\usepackage{graphicx}

\usepackage{amssymb}
\usepackage{amsmath}
\usepackage{amsthm}
\usepackage[all]{xy}
\usepackage{authblk}

\newtheorem{thm}{Theorem}[section]

\newtheorem{prop}[thm]{Proposition}
\theoremstyle{definition}

\newtheorem{Lemma}[thm]{Lemma}
\newtheorem{Theorem}[thm]{Theorem}
\newtheorem{Proposition}[thm]{Proposition}
\newtheorem{Corollary}[thm]{Corollary}
\newtheorem{Definition}[thm]{Definition}
\newtheorem{Example}[thm]{Example}

\theoremstyle{remark}

\numberwithin{equation}{section}
\newcommand{\zhclass}[1]{\langle#1\rangle}

\newcommand{\myd}[1]{D_{(#1)}}

\newcommand{\rdto}[1]{\stackrel{#1}{\longrightarrow}}
\newcommand{\rdtop}[1]{\stackrel{#1}{\longrightarrow_+}}
\newcommand{\rdtopp}[1]{\stackrel{#1}{\longrightarrow_{++}}}

\newcommand{\mycomp}{\mathrm{Comp}}

\def\cvp{\mathrm{CVP}}

\def\deck{\mathrm{D}}
\newcommand{\myim}{\mathrm{Im}}
\newcommand{\myi}{\mathrm{I}}
\newcommand{\mysi}{\mathrm{SI}}

\def\gr{Gr\"{o}bner-Shirshov\ }
\newcommand{\comp}{\mathrm{Comp}}

\newcommand{\lt}{\,\mathrm{lt}}
\newcommand{\lc}{\,\mathrm{lc}}
\newcommand{\lw}{\,\mathrm{lw}}

\newcommand{\dg}[1]{\Delta_{{#1}^2}}

\newcommand{\mycl}{
\qbezier(30.,0)(30.,1.58)(29.83,3.16)
\qbezier(29.83,3.16)(29.67,4.73)(29.33,6.28)
\qbezier(29.33,6.28)(29.,7.83)(28.51,9.34)
\qbezier(28.51,9.34)(28.02,10.84)(27.37,12.29)
\qbezier(27.37,12.29)(26.72,13.73)(25.92,15.1)
\qbezier(25.92,15.1)(25.12,16.47)(24.19,17.75)
\qbezier(24.19,17.75)(23.25,19.03)(22.18,20.2)
\qbezier(22.18,20.2)(21.12,21.37)(19.93,22.42)
\qbezier(19.93,22.42)(18.75,23.47)(17.46,24.4)
\qbezier(17.46,24.4)(16.17,25.32)(14.79,26.1)
\qbezier(14.79,26.1)(13.41,26.88)(11.96,27.51)
\qbezier(11.96,27.51)(10.51,28.14)(9.,28.62)
\qbezier(9.,28.62)(7.49,29.09)(5.93,29.41)
\qbezier(5.93,29.41)(4.38,29.72)(2.8,29.87)
\qbezier(2.8,29.87)(1.23,30.02)(-0.36,30.)
\qbezier(-0.36,30.)(-1.94,29.98)(-3.51,29.79)
\qbezier(-3.51,29.79)(-5.09,29.61)(-6.63,29.26)
\qbezier(-6.63,29.26)(-8.18,28.91)(-9.68,28.4)
\qbezier(-9.68,28.4)(-11.18,27.89)(-12.61,27.22)
\qbezier(-12.61,27.22)(-14.05,26.55)(-15.41,25.74)
\qbezier(-15.41,25.74)(-16.77,24.93)(-18.04,23.97)
\qbezier(-18.04,23.97)(-19.3,23.02)(-20.46,21.94)
\qbezier(-20.46,21.94)(-21.62,20.86)(-22.66,19.66)
\qbezier(-22.66,19.66)(-23.7,18.47)(-24.6,17.17)
\qbezier(-24.6,17.17)(-25.51,15.87)(-26.27,14.48)
\qbezier(-26.27,14.48)(-27.04,13.09)(-27.65,11.63)
\qbezier(-27.65,11.63)(-28.27,10.17)(-28.72,8.66)
\qbezier(-28.72,8.66)(-29.18,7.14)(-29.48,5.58)
\qbezier(-29.48,5.58)(-29.77,4.03)(-29.9,2.45)
\qbezier(-29.9,2.45)(-30.03,0.87)(-29.99,-0.72)
\qbezier(-29.99,-0.72)(-29.95,-2.3)(-29.75,-3.87)
\qbezier(-29.75,-3.87)(-29.54,-5.44)(-29.18,-6.98)
\qbezier(-29.18,-6.98)(-28.81,-8.52)(-28.28,-10.02)
\qbezier(-28.28,-10.02)(-27.75,-11.51)(-27.07,-12.94)
\qbezier(-27.07,-12.94)(-26.38,-14.37)(-25.55,-15.72)
\qbezier(-25.55,-15.72)(-24.72,-17.07)(-23.76,-18.32)
\qbezier(-23.76,-18.32)(-22.79,-19.58)(-21.69,-20.72)
\qbezier(-21.69,-20.72)(-20.6,-21.87)(-19.39,-22.89)
\qbezier(-19.39,-22.89)(-18.18,-23.91)(-16.87,-24.81)
\qbezier(-16.87,-24.81)(-15.56,-25.7)(-14.17,-26.44)
\qbezier(-14.17,-26.44)(-12.77,-27.19)(-11.3,-27.79)
\qbezier(-11.3,-27.79)(-9.84,-28.39)(-8.31,-28.83)
\qbezier(-8.31,-28.83)(-6.79,-29.26)(-5.23,-29.54)
\qbezier(-5.23,-29.54)(-3.67,-29.82)(-2.09,-29.93)
\qbezier(-2.09,-29.93)(-0.51,-30.04)(1.07,-29.98)
\qbezier(1.07,-29.98)(2.66,-29.92)(4.22,-29.7)
\qbezier(4.22,-29.7)(5.79,-29.48)(7.33,-29.09)
\qbezier(7.33,-29.09)(8.87,-28.7)(10.35,-28.16)
\qbezier(10.35,-28.16)(11.84,-27.61)(13.26,-26.91)
\qbezier(13.26,-26.91)(14.68,-26.21)(16.02,-25.36)
\qbezier(16.02,-25.36)(17.36,-24.52)(18.6,-23.54)
\qbezier(18.6,-23.54)(19.85,-22.55)(20.98,-21.45)
\qbezier(20.98,-21.45)(22.11,-20.34)(23.12,-19.12)
\qbezier(23.12,-19.12)(24.13,-17.9)(25.,-16.58)
\qbezier(25.,-16.58)(25.88,-15.26)(26.61,-13.85)
\qbezier(26.61,-13.85)(27.34,-12.45)(27.92,-10.97)
\qbezier(27.92,-10.97)(28.5,-9.5)(28.92,-7.97)
\qbezier(28.92,-7.97)(29.34,-6.44)(29.6,-4.88)
\qbezier(29.6,-4.88)(29.86,-3.32)(29.95,-1.73)
\qbezier(29.95,-1.73)(30.01,-0.72)(30.,0.3)
}

\title{Geometric intersections of loops on surfaces
\footnote{This work was supported by the Natural Science Foundation of Beijing (No. 1202007), and NSF of China (Nos. 12071309, 11961131004).}}

\author[a]{Ying Gu}

\author[b]{Xuezhi Zhao}

\affil[a]{
Institute of Mathematics and Physics, Beijing Union University, Beijing 100101, China}

\affil[b]{
Department of Mathematics, Capital Normal University, Beijing 100048, China}

\begin{document}

\maketitle

\begin{abstract}
Based on Nielsen fixed point theory and Gr\"{o}bner-Shirshov basis, we obtain a simple approach to compute geometric intersection numbers and self-intersection geometric numbers of loops on surfaces.

AMS Subject classification (2010): {55M20, 16S15, 13P10, 57M05}

Keywords: {loops on surfaces, geometric intersections, Gr\"{o}bner-Shirshov basis}

\end{abstract}

\section{Introduction}

Given two loops $\varphi$ and $\psi$ on a compact surface $F$, it is natural to ask: what is the minimal intersection number of loops in the homotopy classes of $\varphi$ and $\psi$? This number is usually called {\em the geometric intersection number}:
$$\myi(\varphi, \psi) = \min \{\sharp (\varphi' \cap \psi') \mid \varphi'\simeq \varphi, \psi'\simeq\psi, \varphi'\pitchfork \psi' \}.$$
Here, $\simeq$ means a free homotopy, and $\pitchfork$ means a transversal intersection. Similarly, we may consider the self-intersection number:
$$\mysi(\varphi) = \min \{\sharp \mbox{ double poins of } \varphi'  \mid \varphi'\simeq \varphi, \varphi'\pitchfork \varphi' \}.$$
Historically, a simple but more essential question was asked: Given a loop $\varphi$ on $F$, can we homotope it into a simple loop? i.e. can we decide whether $\mysi(\varphi)=0$ or not?

This kind of problem was first considered by M. Dehn \cite{Dehn}. Much later, B. Reinhart \cite{Reinhart1962} obtained an algorithm to decide if a loop class on a compact surface contains a simple loop, and also an algorithm to count the minimal intersection number of two loops on a surface.  The key point is the observation: a loop on the surface is simple if and only if any two liftings on the universal covering, which is the Poincar\'{e} disk, have no separating end-points on the circle of infinity. The method to compute the positions of end-points of lifting curves in \cite{Reinhart1962} is essentially a numerical one, and hence it is difficult to make such a computation when the lengths of loops are large. After that people still try to ask for more efficient way to understand the geometric intersection numbers of loops, such as \cite{Hass-Scott} and \cite{Paterson}. Some results about relations between lengths and self-intersections can be found in \cite{Humphries}.

Another algorithm was described by D. Chillingworth \cite{Ch1, Ch2, Ch3}, based on the reduction procedure on elements in $\pi_1(F)$ given by H. Zieschang \cite{Zieschang1965} and winding number introduced by B. Reinhart \cite{Reinhart1960}. In this direction, J. Birman and C. Series \cite{BS1984, BS1985} deal with more general curves.  M. Cohen and M. Lustig \cite{Marshall-Lustig1987, Lustig1987II} and S. P. Tan \cite{Tan1996} gave several combinatorial algorithms to determine minimal geometric intersection numbers of two loops. There is some overlap between the last two works. Hyperbolic
geometry in dimension $2$ was used in a significant way because the minimal intersection happens on the geodesic loops or their small  perturbation. Here, people mainly deal with the surface $F$ having non-empty boundary, where the minimal representation in $\pi_1(F)$ of a loop is unique because $\pi_1(F)$ is a free group. Even very recently, there are still some new algorithms, see \cite{Daci} and \cite{DL2019}. Moreover, Arettines \cite{Arettines2015} presents a purely combinatorial algorithm which produces a representative of given homotopy class with the minimal self-intersection.

The computation of self-intersections of loops on closed surface becomes more difficult. The reason is that the fundamental group  is no longer free.
One attempt at such a computation is contained in \cite{Lustig1987II} giving a
combinatorial algorithm about geometric intersection numbers of arbitrary two loops on an orientable closed surface. The computation of self-intersection in this situation seems to be unclear.

In stead of giving a precise value, people also try to look for some accessible invariants to estimate geometric self-intersection number. One such invariant is the so-called Goldman bracket, which is a structure on the vector space generated by the free homotopy classes of oriented loops on an oriented surface, see \cite{Goldman1986}. M. Chas \cite{Chas2004, Chas2010} gives a combinatorial group theory description of the terms of this kind of Lie bracket and proves that the bracket of two loops has as many terms, counted with multiplicity, as the minimal number of their intersection points. Soon after that, she proved that if a class is chosen at random among all classes of $m$ letters, then for very large $m$ the distribution of the self-intersection number approaches the Gaussian distribution \cite{Chas2012}. She also noticed J. Birman's observation in \cite{BS1985}: very few closed geodesics are simple, see also \cite{Rivin}. Of course, the surface in consideration is assumed to have boundary, and hence as an element in free group, the number of letters of a word makes sense. This new and large
scale treatment may be one source of the significant work of M. Mirzakhani on the asymptotic of rate of the number of simple loops and that of other loops \cite{Mirzakhan2008}. Thus, the determination of intersections of loops on a surface is a classical problem in early stage of geometric topology, is easy to be understood but a little hard to handle, and also has very close relation with modern mathematics.


In this paper, we will present a systematic and straightforward method for determining geometric intersection numbers and self-intersection numbers of loops on oriented closed surfaces. Nielsen fixed point theory and \gr basis are two aspects of our integration.

Nielsen fixed point theory was named after its founder Jakob Nielsen. In 1921, he obtained the minimal number of fixed points in any isotopy class of self-homeomorphisms of the torus.
According to the behavior of lifting of homeomorphisms on the universal covering space, he classified fixed points into various classes. The crucial problem is the estimation of number of fixed points in a given homotopy class of self-maps, see \cite{Jiangbook}. Based on our prior work in \cite{guying}, we put the intersections of loops into generalized Nielsen theory, and obtain a necessary and sufficient condition that two intersections can be cancelled or combined with each other.

Gr\"{o}bner bases were introduced in 1965, together with an algorithm to compute them, by B. Buchberger in his Ph.D. thesis \cite{Buchberger}. He named them after his advisor W. Gr\"{o}bner. Now it becomes one of the main practical tools for solving systems of polynomial equations and computing the images of algebraic varieties. Since the natural connection between membership of an ideal and the word problem for a group, Gr\"{o}bner base was applied into group theory, and was named as ``rewriting system'', see \cite{Madlener1989} for more details. Some people use the notation \gr basis to indicate the non-commutative version, because the work of Shirshov \cite{Shirshov}. When we are given a presentation of group, there is no finite \gr basis in general, and hence Buchberger's algorithm may not terminate.

In this paper we obtain a \gr basis, written as $D$, of the closed orientable surface group in a special presentation. Each loop can be written as a cyclically $D$-reduced word, which can be regarded as an algebraic version of a geodesic loop. If two words determining two loops on a surface are given, one can obtain the corresponding cyclically $D$-reduced words by a classical reduction procedure. By using Nielsen theory, we show that the geometric intersection numbers can be read  directly from two cyclically $D$-reduced words. The self-intersection number of a loop can be obtained in a similar way. Thus, we give a clear answer of the long-standing question: how to compute the minimal intersections and self-intersections of loop classes on closed orientable surfaces. Our method still works for situation of a surface with boundary, but more techniques seem to be needed in the situation of a closed non-orientable surface.

Our paper is organized as follows. In section 2, we convert geometric intersections into common value pairs (the pre-image of intersections). With the help of hyperbolic geometry, we give a description of the set of common value pairs of two piecewise-geodesic loops on the surface. By using a special generating set for the surface group, we obtain a \gr basis in section 3. In section 4, we shall show that if loops are piecewise-geodesic, then the numbers of common value classes  can be determined based on $D$-cyclical reduction of our \gr basis. Section 5 deals with the indices of common value classes. The numbers of common value classes with non-zero indices can be compute for all loops. This number is proved, in Section 6, to be exactly the geometric intersection number of two loops. Similar results for geometric self-intersection are also given. In final section, we give an example to explain our practical computation.


\section{Intersections of loops on surfaces}

In this section, we recall some basic materials about $2$-dimensional hyperbolic geometry. Some ideas of common value pairs are also given, especially in the case of loops on surfaces. We fix some notations for further use in this paper.

Let $\varphi, \psi\colon X\to Y$ be two maps. The set of common value pairs $\cvp(\varphi, \psi)$ of $\varphi$ and $\psi$ is defined to be the set
$$(\varphi\times \psi)^{-1}(\dg{Y})=\{(u, v)\in X^2 \mid \varphi(u) = \psi(v)\}$$
(See \cite[Definition 4.1]{guying}), where $\dg{Y}$ is the diagonal of $Y^2$.  We assume that $X$ and $Y$ have  their universal covering $p_X\colon \tilde X\to X$ and $p_Y\colon \tilde Y\to Y$, respectively. We write $\deck(\tilde X)$ and $\deck(\tilde Y)$ for the deck transformation groups of two universal coverings.

\begin{Proposition}(See \cite[Proposition 4.2, 4.6]{guying})\label{prop-GZ}
Fix a lifting $\tilde \varphi$ of $\varphi$ and a lifting $\tilde \psi$ of $\psi$. Then the set $\cvp(\varphi, \psi)$ of common value pairs of $\varphi$ and $\psi$ is equal to $\cup_{\gamma \in \deck(\tilde Y)} p_X\times p_X (\cvp(\tilde \varphi, \gamma \tilde \psi))$, and for any two elements $\alpha, \beta \in \deck(\tilde Y)$, the following three statements are equivalent:

(1) $p_X\times p_X (\cvp(\tilde \varphi, \alpha \tilde \psi))\cap p_X\times p_X (\cvp(\tilde \varphi, \beta \tilde \psi))\ne \emptyset$;

(2) $p_X\times p_X (\cvp(\tilde \varphi, \alpha \tilde \psi)) = p_X\times p_X (\cvp(\tilde \varphi, \beta \tilde \psi))\ne \emptyset$;

(3) $\beta = \tilde \varphi_{D}(\delta)\alpha \tilde \psi_{D}(\varepsilon)$ for some $\delta, \varepsilon\in \deck(\tilde X)$.
\end{Proposition}
The homomorphism $\tilde \varphi_{D}\colon \deck(\tilde X)\to \deck(\tilde Y)$ is  determined by the relation $\tilde \varphi(\eta \tilde x) = \tilde \varphi_{D}(\eta)\tilde \varphi(\tilde x)$ for all $\eta\in \deck(\tilde X)$ and $\tilde x\in \tilde X$, and $\tilde \psi_{D}$ is defined similarly (see \cite[Ch. III]{Jiangbook} for more details). If $\varphi=\psi$, we shall say self-common value pairs, self-common value classes, etc.

It should be mentioned that the set $\cvp(\varphi, \psi)$  of common value pairs of $\varphi$ and $\psi$ is a disjoint union of common value classes. A non-empty subset $p_X\times p_X (\cvp(\tilde \varphi, \alpha \tilde \psi))$ of $\cvp(\varphi, \psi)$ is said to be the  common value class determined by $(\tilde \varphi, \alpha \tilde \psi)$, or by $\alpha$ if two liftings $\tilde \varphi$ and $\tilde \psi$ are clearly chosen in advance.

Now, we consider a special case of common value pairs, where $\varphi$ and $\psi$ are maps from the circle $S^1$ to an orientable closed surface $F_g$ of genus $g\ge 2$. Recall that

\begin{Lemma}(see \cite[\S 1]{Casson} and \cite[\S 20]{GTM104})
There is a canonical isomorphism from $\mathrm{PSL}(2, \mathbb{\mathbb{R}})$ to orientation-preserving isometry group of $\mathbb{H}^2$, which is given by
$$M=\left(
\begin{array}{cc}
a & b\\
c & d\\
\end{array}
\right)
\mapsto
f_M(z) = \frac{az+b}{cz+d}.
$$
Moreover, we have that for any two elements $M$ and $M'$ in $\mathrm{PSL}(2, \mathbb{\mathbb{R}})$, $f_{M'}f_{M} = f_{MM'}$.
\end{Lemma}

Using the classical identification of the fundamental group of a space and the deck transformation group of its universal covering, we have

\begin{Lemma}\label{pi1}(see \cite[\S 20]{GTM104})
The closed oriented surface $F_g$ of genus $g\ge 2$ can be regarded as the orbit space $\mathbb{H}^2/\Gamma_g$, where $\Gamma_g$ is a discrete subgroup of $\mathrm{Iso}^+(\mathbb{H}^2)$  generated by $\delta_1, \delta_2, \ldots, \delta_{2g}$, and each $\delta_j$ is the hyperbolic translation determined by $M_g^{-j}Q_gM_g^j$, in which
$$
M_g=\left(
\begin{array}{cc}
  \cos \frac{(2g-1)\pi}{4g} & \sin \frac{(2g-1)\pi}{4g}\\
 -\sin \frac{(2g-1)\pi}{4g} & \cos \frac{(2g-1)\pi}{4g}\\
\end{array}
\right),
\
Q_g=
\left(
\begin{array}{cc}
  \frac{\cos \frac{\pi}{4g} +  \sqrt{\cos \frac{\pi}{2g}}}{\sin\frac{\pi}{4g}} & 0\\
   0& \frac{\sin\frac{\pi}{4g}}{\cos \frac{\pi}{4g} +  \sqrt{\cos \frac{\pi}{2g}}}
\end{array}
\right).
$$
Hence, the fundamental group $\pi_1(F_g, y_0) $ of $F_g$ has following presentation
\begin{equation}\label{presentation}
\pi_1(F_g, y_0) = \langle c_1, c_2, \ldots, c_{2g}\mid c_1 c_2 \cdots c_{2g} = c_{2g} \cdots c_2 c_1\rangle,
\end{equation}
where $c_j$ is the loop class determined by $p(\omega_j)$ for $j=1,2,\ldots, 2g$, in which $\omega_j$ is the unique geodesic from $\tilde y_0=i$ to $\delta_j(\tilde y_0)$, $p_{\mathbb{H}^2}\colon \mathbb{H}^2\to \mathbb{H}^2/\Gamma_g=F_g$ is the universal covering map and $y_0=p_{\mathbb{H}^2}(\tilde y_0)$.
\end{Lemma}

In pure group theory, an isomorphism from the canonical presentation $\langle a_1, b_1,\ldots, a_g, b_g\mid[a_1,b_1]\cdots[a_g,b_g]=1\rangle$ of fundamental group of $F_g$ to the presentation (\ref{presentation}) can be defined by
$$\begin{array}{l}
a_j\mapsto c_{2j-1}\cdots c_{2g}c_{2j-1}^{-1},\
b_j\mapsto c_{2j+1}\cdots c_{2g}c_{2j-1}^{-1},\ \ j=1,2,\ldots, g-1,\\
a_g\mapsto c_{2g-1},\ \ b_g\mapsto c_{2g}.
  \end{array}
$$

By a concrete computation, in Poincar\'{e} model $\mathbb{D}^2$, each hyperbolic translation $\delta_j$ corresponding to the generator $c_j$ of $\pi_1(F_g, y_0)$ in presentation (\ref{presentation}) has a diameter as its translation axis with an attracting fixed point $(-1)^{j}e^{\frac{(j-1)\pi}{2g}i}$ and an expanding fixed point $(-1)^{j-1}e^{\frac{(j-1)\pi}{2g}i}$.
We write \begin{equation}\label{eq-T}
T(c_j)=:(-1)^{j}e^{\frac{(j-1)\pi}{2g}i},\ \ T(c_j^{-1})=:(-1)^{j+1}e^{\frac{(j-1)\pi}{2g}i},
\end{equation}
which are also two ending points on the circle at infinity of the following lifting $\tilde c_j$ of the loop $c_j$.
$$\xymatrix{
 (\mathbb{R}^1, 0) \ar[r]^{\tilde c_j} \ar[d]_{p_{S^1}} & (\mathbb{D}^2, \tilde y_0=0)\ar[d]^p \\
 (S^1, e^0) \ar[r]^{c_j} & (F_g, y_0).
 }
$$
Take $g=2$ as an example, we have the following:
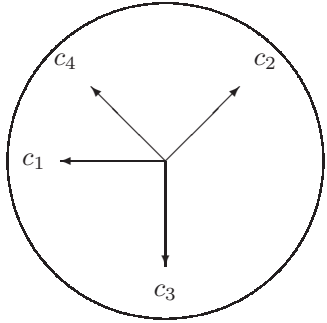
\begin{center}
\setlength{\unitlength}{0.7mm}
\begin{picture}(60,70)(-30,-35)
\mycl
\put(0,0){\vector(-1,0){20}}
\put(-30,0){\line(1,0){60}}
\put(0,0){\vector(1,1){14.1420}}
\put(0,0){\line(1,1){21.2132}}
\put(0,0){\line(-1,-1){21.2132}}
\put(0,0){\vector(0,-1){20}}
\put(0,-30,0){\line(0,1){60}}
\put(0,0){\vector(-1,1){14.1420}}
\put(0,0){\line(-1,1){21.2132}}
\put(0,0){\line(1,-1){21.2132}}

\put(-25,3){\makebox(0,0)[cc]{$\tilde c_1$}}
\put(12,17){\makebox(0,0)[cc]{$\tilde c_2$}}
\put(3,-25){\makebox(0,0)[cc]{$\tilde c_3$}}
\put(-15,19){\makebox(0,0)[cc]{$\tilde c_4$}}
\put(20,-8){\makebox(0,0)[cc]{$\mathbb{D}^2$}}
\put(8,3){\makebox(0,0)[cc]{$\tilde y_0$}}
\put(27,27){\makebox(0,0)[cc]{$T(c_2)$}}
\put(-27,-27){\makebox(0,0)[cc]{$T(c_2^{-1})$}}
\put(21.2132,21.2132){\makebox(0,0)[cc]{$\bullet$}}
\put(-21.2132,-21.2132){\makebox(0,0)[cc]{$\bullet$}}
\put(-27,27){\makebox(0,0)[cc]{$T(c_4)$}}
\put(27,-27){\makebox(0,0)[cc]{$T(c_4^{-1})$}}
\put(-21.2132,21.2132){\makebox(0,0)[cc]{$\bullet$}}
\put(21.2132,-21.2132){\makebox(0,0)[cc]{$\bullet$}}
\put(32,0){\makebox(0,0)[lc]{$T(c_1)$}}
\put(-32,0){\makebox(0,0)[rc]{$T(c_1^{-1})$}}
\put(-30,0){\makebox(0,0)[cc]{$\bullet$}}
\put(30,0){\makebox(0,0)[cc]{$\bullet$}}
\put(0,-32){\makebox(0,0)[ct]{$T(c_3)$}}
\put(0,32){\makebox(0,0)[cb]{$T(c_3^{-1})$}}
\put(0,-30){\makebox(0,0)[cc]{$\bullet$}}
\put(0,30){\makebox(0,0)[cc]{$\bullet$}}
\end{picture}
\end{center}
The relative positions of all $T(\cdot)$'s indicate that those of all generator loops around the base point of $F_g$.

It is obvious that

\begin{Proposition}
Let $w', w''\in \{c_1, c^{-1}_1, \ldots, c_{2g}, c^{-1}_{2g}\}$. Then (1) $w'=w''$ if and only if $T(w')=T(w'')$, (2) $w'^{-1}=w''$ if and only if $-T(w')=T(w'')$.
\end{Proposition}


It is well-known that the set of loop classes in $F_g$ is in one-to-one correspondence with the set of  conjugacy classes in $\pi_1(F_g)$. For our purpose, we use a special kind loops if an element $\pi_1(F_g)$ is given, see bellow.

\begin{Lemma}\label{piecewise-geodesic}
Let $u_1\cdots u_m$ be a word in the letter set $\{c_1, c^{-1}_1,\ldots c_{2g}, c^{-1}_{2g}\}$, giving a non-trivial element in $\pi_1(F_g, y_0)$ according to the presentation  in (\ref{presentation}). Then there is a unique map $\varphi: S^1 \to F_g$, said to be a {\em piecewise-geodesic loop}, such that $[0,1]\ni \lambda \mapsto \varphi(e^{\frac{2\pi(k+\lambda-1)i}{m}})\in F_g$ is the unique geodesic representing $u_k$ for $k=1,2,\ldots, m$, and such that $|\frac{\mathrm{d} \varphi(e^{\frac{2\pi(k+\lambda -1)i}{m}})}{\mathrm{d}\, \lambda}|$ is constant.
\end{Lemma}

\begin{proof}
It is sufficient to prove the uniqueness, which follows directly from hyperbolicy of $F_g$.
\end{proof}

Since common value pair theory, as a kind of generalization of Nielsen fixed point theory, deals with homotopy invariant, it is sufficient to consider the piecewise-geodesic loops mentioned above. Since each geodesic section is one of $2g$ loops with base point $y_0$, the intersection of any two piecewise-geodesic loops is either the singleton $\{y_0\}$ or a union of geodesic loops representing some generators. Next definition and theorem will describe the pre-image of these intersections, i. e. the set of common value pairs.

\begin{Definition}\label{def-standard-lifting}
Let $\varphi: S^1\to F_g$ be a loop at base point $y_0 = p(0)\in F_g$. A lifting $\tilde \varphi_S: \mathbb{R}^1\to \mathbb{D}^2$ is said to be {\em standard lifting} of $\varphi$ if it fits into:
$$\xymatrix{
 (\mathbb{R}^1, 0) \ar[r]^{\tilde \varphi_S} \ar[d]_{p_{S^1}} & (\mathbb{D}^2, \tilde y_0=0)\ar[d]^p \\
 (S^1, e^0) \ar[r]^{\varphi} & (F_g, y_0)
 }
$$
\end{Definition}


\begin{Theorem}\label{th:cvp-component}
Let $\varphi,\psi: S^1\to F_g$ be two piecewise-geodesic loops which are determined by $u_1\cdots u_m$ and $v_1\cdots v_n$, respectively, where $u_k, v_l\in \{c_1, c^{-1}_1, \ldots, c_{2g}, c^{-1}_{2g}\}$. Then each component of common value set $\cvp(\varphi, \psi)$  is of one following type:
\begin{itemize}
  \item[(1)]  $\{( e^{\frac{2\pi k i}{m}}, e^{\frac{2\pi l i}{n}})\}$,
  \item[(2)]  $\{( e^{\frac{2\pi (k+\lambda) i}{m}}, e^{\frac{2\pi (l+\lambda) i}{n}})\mid 0\le \lambda \le q\}$ for some positive integer $q$,
  \item[(3)]  $\{( e^{\frac{2\pi (k+\lambda) i}{m}}, e^{\frac{2\pi (l+q-\lambda) i}{n}})\mid 0\le \lambda \le q\}$ for some positive integer $q$,
  \item[(4)]  $\{( e^{\frac{2\pi (k+\lambda) i}{m}}, e^{\frac{2\pi (l+\lambda) i}{n}})\mid -\infty <\lambda < +\infty\}$,
  \item[(5)]  $\{( e^{\frac{2\pi (k+\lambda) i}{m}}, e^{\frac{2\pi (l-\lambda) i}{n}})\mid -\infty <\lambda < +\infty\}$,
\end{itemize}
where $k$ and $l$ are integers. Moreover, the common value class containing the pair $( e^{\frac{2\pi k i}{m}}, e^{\frac{2\pi l i}{n}})$ is determined by $((u_1\cdots u_k)^{-1}\tilde \varphi_S, (v_1\cdots v_l)^{-1}\tilde \psi_S)$, where $\tilde \varphi_S$ and $\tilde \psi_S$ are respectively standard liftings of $\varphi$ and $\psi$, and $u_p, v_q$ are regarded as hyperbolic translations in $\mathbb{D}^2$.
\end{Theorem}

\begin{proof}
By definition of standard liftings, we have that $\tilde \varphi_S(\frac{k}{m}) = u_1\cdots u_k(\tilde y_0)$ and $\tilde \psi_S(\frac{l}{n}) =  v_1\cdots v_l(\tilde y_0)$ for $k=1,2,\ldots, m$ and $l =1,2,\ldots, n$. Hence,
$$
(\frac{k}{m}, \frac{l}{n}) \in \cvp((u_1\cdots u_k)^{-1}\tilde \varphi_S,\ (v_1\cdots v_l)^{-1}\tilde \psi_S).
$$
Hence,
$$
(e^{\frac{2\pi k i}{m}},  e^{\frac{2\pi l i}{n}})
 \in
  (p_{S^1}\times p_{S^1})(\cvp((u_1\cdots u_k)^{-1}\tilde \varphi_S,\
  (v_1\cdots v_l)^{-1}\tilde \psi_S)),
$$
and therefore we obtain  our first type if the pair $(e^{\frac{2\pi k i}{m}},  e^{\frac{2\pi l i}{n}})$ is isolated in $S^1\times S^1$.

Note that $\varphi(e^{\frac{2\pi k i}{m}})=\psi(e^{\frac{2\pi l i}{n}})=y_0$. Each pair $(e^{\frac{2\pi k i}{m}}, e^{\frac{2\pi l i}{n}})$ is a common value pair of $\varphi$ and $\psi$ for all integers $k$ and $l$. Suppose that there is a common value pair $(e^{\frac{2\pi \lambda' i}{m}},  e^{\frac{2\pi \lambda'' i}{n}})$ of $\varphi$  and $\psi$, where $\lambda'$ or $\lambda''$ is not an integer. This implies that $\varphi(e^{\frac{2\pi \lambda' i}{m}})$ or $\psi(e^{\frac{2\pi \lambda'' i}{n}})$  is not the base point $y_0$. By definition $\varphi(e^{\frac{2\pi \lambda' i}{m}})=\psi(e^{\frac{2\pi \lambda'' i}{n}})$, this common point is not the base point $y_0$.
Thus, both of $\lambda'$ and $\lambda''$ are not integers.

We can write $\lambda' \in (k,k+1)$ and $\lambda''\in (l, l+1)$ for some integers $k$ and $l$. By the uniqueness of geodesic, we have $\varphi(\{e^{\frac{2\pi \lambda i}{m}} \mid \lambda\in [k,k+1] \})=\psi(\{e^{\frac{2\pi \lambda i}{n}} \mid \lambda\in [l,l+1] \})$. Recall from Lemma~\ref{piecewise-geodesic} that $|\frac{\mathrm{d} \varphi(e^{\frac{2\pi \lambda i}{m}})}{\mathrm{d}\, \lambda}| = |\frac{\mathrm{d} \psi(e^{\frac{2\pi \lambda i}{n}})}{\mathrm{d}\, \lambda}|$.    There are only two possibilities:
$$\mbox{(i) } \varphi(e^{\frac{2\pi (k+\varepsilon) i}{m}})
     = \psi(e^{\frac{2\pi (l+\varepsilon) i}{n}}),\ \
 \mbox{(ii) } \varphi(e^{\frac{2\pi (k+\varepsilon) i}{m}})
    = \psi(e^{\frac{2\pi (l+1-\varepsilon) i}{n}})$$
for all $\varepsilon\in [0,1]$. Thus, the common value pair $( e^{\frac{2\pi \lambda' i}{n}}, e^{\frac{2\pi \lambda'' i}{m}})$ has the same component with $( e^{\frac{2\pi l i}{n}}, e^{\frac{2\pi k i}{m}})$ or $( e^{\frac{2\pi (l+1) i}{n}}, e^{\frac{2\pi k i}{m}})$. Thus any component of $\cvp(\varphi, \psi)$ is a union of several successive sets of type (i) or those of type (ii), which are respectively in second and third type of components of common value subsets if the length of extended range of $\varepsilon$ (or say the number of sets involved) is bounded.

Now, we consider the case where the range of $\varepsilon$ is unbounded. Without loss of generality, we assume that $\varphi(e^{\frac{2\pi (k+\varepsilon) i}{m}}) = \psi(e^{\frac{2\pi (l+\varepsilon) i}{n}})$ for all $\varepsilon\ge 0$. Note  that $e^{\frac{2\pi (k+\varepsilon) i}{m}}= e^{\frac{2\pi (k+\varepsilon-mn) i}{m}}$ and that $e^{\frac{2\pi (l+\varepsilon) i}{n}} = e^{\frac{2\pi (l+\varepsilon-mn) i}{n}}$. We obtain that $\varphi(e^{\frac{2\pi (k+\varepsilon) i}{m}}) = \psi(e^{\frac{2\pi (l+\varepsilon) i}{n}})$ for all $\varepsilon$. This is the type (4). The type (5) can be obtained similarly.
\end{proof}

From the proof we also obtain that

\begin{Corollary}
Corresponding to the five types of components of $\cvp(\varphi, \psi)$ in above theorem, the words $u_1\cdots u_m$ and $v_1\cdots v_n$ satisfy following properties:
\begin{itemize}
  \item[(1)]  $u_k\ne v_l$, $u_{k+1}\ne v_{l+1}$, $v_l^{-1}\ne u_{k+1}$ and $u_k\ne v_{l+1}^{-1}$,

  \item[(2)]  $u_{k+r}=v_{l+r}$ for $r=1, \ldots, q$, $u_{k}\ne v_{l}$ and $u_{k+q+1}\ne v_{l+q+1}$,

  \item[(3)]  $u_{k+r}= v_{l+q-r}^{-1}$ for $r=1, \ldots, q$, $u_{k}\ne v_{l+q+1}^{-1}$ and $u_{k+q+1}\ne v_{l}^{-1}$,

  \item[(4)]  $u_{k+r}=v_{l+r}$ for all integer $r$,

  \item[(5)]  $u_{k+r}= v_{l-r}^{-1}$ for all integer $r$.
\end{itemize}
Here, the subscripts of letters $u$ and $v$ will be regarded as ones module $m$ and $n$, respectively.
\end{Corollary}

We shall write $(k,l,0)_{\mu, \nu}$, $(k,l,q)_{\mu, \nu}$ and $(k,l,-q)_{\mu, \nu}$ for  the components of types (1), (2) and (3) mentioned in Theorem~\ref{th:cvp-component} and its Corollary, respectively.

A general theory tells us the number of the essential common value classes gives a lower bound for the number of geometric intersections, see \cite[Theorem 4.10]{guying}. In general, common value pairs in different components of $\cvp(\varphi, \psi)$ may lie in the same common value class. Thus, we need  to determine if there are some classes which coincide with each other among these, at most, $m\times n$ candidates.  We also need to see if a common value class can be removed by a homotopy, which will be considered in section 5.

\section{\gr basis}

In this section, we shall give a \gr basis of surface group $\pi_1(F_g)$ in the presentation (\ref{presentation}). Our approach here is similar to that of \cite{Vesnin}. We review some basic results about \gr bases.

\def\zhord{\succ}

Let $X=\{x_1, x_2, \ldots, x_m\}$ be a linearly ordered set, $K$ be a field, and $K\zhclass{X}$ be the free associative and non-commutative algebra over $X$ with coefficient $K$. On the set of words consisting of letters in $X$ we impose a well order ``$\zhord$'' that is compatible with the cancelations
of words. For example, it may be the length-lexicographical order, i.e. we say $\alpha\zhord\beta$ if either $|\alpha|>|\beta|$ or $|\alpha|=|\beta|$, $\alpha=\gamma x_i \zeta$, $\beta=\gamma x_j\eta$  and $x_i\zhord x_j$ according to the given order in $X$.

Each element $f$ in $K\zhclass{X}$ is said to be a polynomial, and is written as $\sum_{i=1}^s k_i\gamma_i$, where $k_i\in K$ and $\gamma_i$ is a word for all $i$. We may arrange these $\gamma_i$'s so that $\gamma_1\zhord\gamma_2\zhord\cdots\zhord\gamma_m$. Then $k_1\gamma_1$, $k_1$, $\gamma_1$ are said to be {\em leading term}, {\em leading coefficient} and {\em leading word}, respectively. They are denoted as $\lt(f)$, $\lc(f)$ and $\lw(f)$, respectively.

We say that {\em $f$ reduces $f'$ to $f''$} (or $f'$ {\em is reduced by $f$ to $f''$}) if $\lw(f')=\gamma\lw(f)\delta$ and $f''=f'-\frac{\lc(f')}{\lc(f)}\gamma f\delta$, where $\gamma$ and $\delta$ are words. We write $f'\rdto{f} f''$. Let $F$ be a set consisting of polynomials. We say {\em $F$ reduces $f'$ to $f''$} (or {\em $f'$ is reduced by $F$ to $f''$}) if there are polynomials $f_j\in F$, $j=1,2,\ldots m$, such that
$$f'\rdto{f_1} f'_1 \rdto{f_2} f'_2 \rdto{f_3} \cdots \rdto{f_m} f'_m=f''. $$
We write $f'\rdtop{F} f''$. Especially, if $f''$ is $F$-irreducible, then we write
$f'\rdtopp{F} f''$. In this case, $f''$ is said to be a {\em $F$-reduced form} of $f$.

Let $f$ and $f'$ be two polynomials. The set $\mycomp(f,f')$ of compositions of $f$ and $f'$ consists of following two parts:
$$
\begin{array}{l}
\{\lc(f') f \delta  - \lc(f) \gamma f'
 \mid
  \lw(f) \delta= \gamma \lw(f')
\},\\
\{\lc(f') f  - \lc(f) \gamma' f' \delta'
  \mid \lw(f) = \gamma' \lw(f') \delta'\}.
\end{array}
$$
The elements in the first set are called {\em compositions of intersection}. A common sub-word $\alpha$ of the leading words of $f$ and $f'$ with $\lw(f)= \beta\alpha$ and $\lw(f') = \alpha\beta'$ for some words $\beta$ and $\beta'$ is said to be an {\em overlap} of them. The elements in the second set are called {\em compositions of including}. It should be mentioned that $\mycomp(f,f')$ may  be different from $\mycomp(f',f)$.

\begin{prop}(see \cite[Sec. 2]{Vesnin})
Let $B$ be a set of polynomials which generate the ideal $I$ of $K\zhclass{X}$. Then following statements are equivalent to each other:
\begin{itemize}
  \item $B$ is a \gr basis of $I$,
  \item a polynomial $f$ lies in $I$ if and only if  $f$ is reduced to $0$ by $B$,
  \item any composition of two polynomials in $B$ is reduced to $0$ by $B$.
\end{itemize}\label{prop-three}
\end{prop}

Let $G$ be a group with presentation $\zhclass{x_1, \ldots, x_m\mid \gamma_1, \ldots, \gamma_n}$. Let $I\subset K\zhclass{x_1, x_1^{-1}, \ldots, x_n, x_n^{-1}}$  be the ideal
generated by
$$\gamma_j-1, j=1, \ldots n, \  \ x_ix_i^{-1}-1,  x^{-1}_ix_i-1,i=1, \ldots, m.$$
This ideal is said to be the ideal of group with respect to given presentation. It is easy to know that two words $\alpha$ and $\beta$ represent the same element in $G$ if and only if $\alpha-\beta\in I$. If $B$ is a \gr basis of $I$, then $\alpha-\beta\in I$ if and only if $\alpha$ and $\beta$ have the same normal form $\gamma$, i.e. $\alpha\rdtopp{B}\gamma$ and $\beta\rdtopp{B}\gamma$.


Now, we shall construct a \gr basis of $\pi_1(F_g)$ by checking the third statement in Proposition~\ref{prop-three}, which is actually the Buchberger algorithm.

\begin{thm}\label{grobner}
Given a presentation $\zhclass{c_1, \ldots, c_{2g}\mid c_{2g}\cdots c_1 =c_1\cdots c_{2g}}$ of the fundamental group of the closed surface of genus $g$ with $g\ge 2$.
By using the length-lexicographical order of generators
$$c_{2g}^{-1}\succ \cdots \succ c_{2}^{-1}  \succ  c_{1}^{-1} \succ  c_{1} \succ  c_{2} \succ \cdots  \succ c_{2g} ,$$
there is a \gr basis $D$, consisting of the following:

(1) $\myd{1,j,s}=c_j (c_{j-1}\cdots c_1c_{2g}^{-1}\cdots c_{j+1}^{-1})^s c_j^{-1} - (c_{j+1}^{-1}\cdots c_{2g}^{-1}c_1 \cdots c_{j-1} )^s$ for $j=2,\ldots, 2g$ and $s=1,2, \ldots$;

(2) $\myd{2,j,s}=c_j(c_{j+1}\cdots c_{2g}c_1^{-1}\cdots c_{j-1}^{-1})^s c_j^{-1} - (c_{j-1}^{-1}\cdots c_1^{-1}c_{2g}\cdots c_{j+1})^s$ for $j=2,\ldots, 2g$ and $s=1,2, \ldots$;

(3) $\myd{3} = c_{2g}^{-1}\cdots c_1^{-1} -  c_1^{-1}\cdots c_{2g}^{-1} $;

(4) $\myd{4} = c_1\cdots c_{2g} - c_{2g}\cdots c_1$;

(5) $\myd{5,i} = c_i^{-1}\cdots c_{2g}^{-1}c_1\cdots c_{i-1} -  c_{i-1}\cdots c_1c_{2g}^{-1}\cdots c_i^{-1}$ for $i=2, \dots, 2g$;

(6) $\myd{6,i} = c_{i-1}^{-1}\cdots c_1^{-1}c_{2g}\cdots c_i - c_i\cdots c_{2g}c_1^{-1}\cdots c_{i-1}^{-1}$ for $i=2, \dots, 2g$;

(7) $\myd{7,i} = c_i^{-1}c_i-1$, for $i=1, \ldots, 2g$;

(8) $\myd{8,i} = c_ic_i^{-1}-1$, for $i=1, \ldots, 2g$.
\end{thm}

\begin{proof}
We write $\myd{k}$, $k=1,2,\ldots 8$, for corresponding subsets of our \gr basis $D$.
Note that the ideal $I$ of the surface group under our consideration in $K\zhclass{c_1, c_1^{-1}, \ldots, c_{2g}, c_{2g}^{-1}}$ is generated by $\myd{4}\cup \myd{7}\cup \myd{8}$. For any $i$, we have $c_{i}^{-1}\cdots c_{2g}^{-1}\myd{4}c_{2g}^{-1}\cdots c_i^{-1}\in I$. Since
$$c_{i}^{-1}\cdots c_{2g}^{-1}\myd{4}c_{2g}^{-1}\cdots c_i^{-1}
\rdtopp{\myd{7}\cup\myd{8}}
\left\{
\begin{array}{ll}
\myd{3} &  \mbox{ if } i=1,\\
\myd{5,i} &  \mbox{ if } 2\le i\le 2g,
\end{array}
\right.
$$
we obtain that both of $\myd{3}$ and $\myd{5}$  are contained in $I$. For $i$ with $2\le i\le 2g$, the reduction
$$c_{i-1}^{-1}\cdots c_{1}^{-1}\myd{4}c_{1}^{-1}\cdots c_{i-1}^{-1}
\rdtopp{\myd{7}\cup\myd{8}}
\myd{6,i}
$$
implies that $\myd{6}$ is contained in $I$.
Since $$
\begin{array}{rcl}
\myd{1,j,s}c_j
& = &
  c_j(c_{j-1}\cdots c_1c_{2g}^{-1}\cdots c_{j+1}^{-1})^s c_j^{-1}c_j
  - (c_{j+1}^{-1}\cdots c_{2g}^{-1}c_1\cdots c_{j-1})^sc_j
\\
&\rdto{\myd{7,j}} &
  c_j(c_{j-1}\cdots c_1c_{2g}^{-1}\cdots c_{j+1}^{-1})^s
  - (c_{j+1}^{-1}\cdots c_{2g}^{-1}c_1\cdots c_{j-1})^sc_j\\
&\rdto{\myd{5,j+1}} &
  c_j(c_{j-1}\cdots c_1c_{2g}^{-1}\cdots c_{j+1}^{-1})^s \\
& & \ \ \  - (c_{j+1}^{-1}\cdots c_{2g}^{-1}c_1\cdots c_{j-1})^{s-1}
  c_jc_{j-1}\cdots c_1c_{2g}^{-1}\cdots c_{j+1}^{-1}\\
& \cdots & \cdots \cdots \\
&\rdto{\myd{5,j+1}} & c_j(c_{j-1}\cdots c_1c_{2g}^{-1}\cdots c_{j+1}^{-1})^s
 -c_j(c_{j-1}\cdots c_1c_{2g}^{-1}\cdots c_{j+1}^{-1})^s \\
& =& 0,
\end{array}
$$
we have that $\myd{1,j,s}c_j\in I$, and therefore $\myd{1,j,s}\in I$. Similarly, we can prove that $\myd{2,j,s}\in I$.

We are going to show that any composition of two polynomials in $D$ is reduced to $0$ by $D$ itself. Clearly, there is not any composition of including. It is sufficient to consider the compositions of intersection. There are $64$ cases in total, which are written as (Cij), where $1\le i,j\le 8$. We drop the cases when the composition sets are empty.



(C13) The composition set $\comp(\myd{1,j,s}, \myd{3})$ is empty except for the case $j=2g$. The composition $\comp(\myd{1,2g,s}, \myd{3})$ contains exactly one element
$$
\myd{1,2g,s} c_{2g-1}^{-1}\cdots c_1^{-1} - c_{2g}(c_{2g-1}\cdots c_1)^s\myd{3}\\
\rdtop{\myd{7}\cup\myd{8}}
-(c_1 \cdots c_{2g-1})^{s-1} + c_{2g}(c_{2g-1}\cdots c_1)^{s-1}c_{2g}^{-1},
$$
which is reduced to $0$ by $\myd{1,2g,s-1}$ if $s>1$ or by $\myd{8,2g}$ if $s=1$.


(C15) The composition set of $\myd{1,j,s}$ and $\myd{5,i}$ is non-empty only if $j=i$. The unique element in  $\comp(\myd{1,j,s}, \myd{5,j})$ is:
$$
\begin{array}{cl}
&
\myd{1,j,s} c_{j+1}^{-1}\cdots c_{2g}^{-1}c_1\cdots c_{j-1}
 - c_{j}(c_{j-1}\cdots c_1c_{2g}^{-1}\cdots c_{j+1}^{-1})^s\myd{5,j} \\
=
 &
-(c_{j-1}^{-1}\cdots c_1^{-1}c_{2g}\cdots c_{j+1})^{s+1} + c_{j}(c_{j+1}\cdots c_{2g}c_1^{-1}\cdots c_{j-1}^{-1})^{s+1}c_j^{-1}\\
\rdto{\myd{2,j,s+1}}
&
0.
\end{array}
$$

(C16) The composition set of $\myd{1,j,s}$ and $\myd{6,i}$ is non-empty only if $i-1\ge j$. In this situation, the overlap of $\myd{1,j,s}$ and $\myd{6,i}$  must be $c_{i-1}^{-1}\cdots c_{j+1}^{-1}c_j^{-1}$. The unique composition is:
$$
\begin{array}{cl}
&
\myd{1,j,s}c_{j-1}^{-1}\cdots c_1^{-1}c_{2g}\cdots c_i - c_j(c_{j-1}\cdots c_1c_{2g}^{-1}\cdots c_{j+1}^{-1})^{s-1}c_{j-1}\cdots c_1c_{2g}^{-1}\cdots c_i^{-1} \myd{6,i}\\
= &
 - (c_{j+1}^{-1}\cdots c_{2g}^{-1}c_1\cdots c_{j-1})^s c_{j-1}^{-1}\cdots c_1^{-1}c_{2g}\cdots c_i \\
&
\ + c_j(c_{j-1}\cdots c_1c_{2g}^{-1}\cdots c_{j+1}^{-1})^{s-1}c_{j-1}\cdots c_1c_{2g}^{-1}\cdots c_i^{-1}
  c_i\cdots c_{2g} c_1^{-1}\cdots c_{i-1}^{-1} \\
\rdtop{\myd{7}\cup \myd{8}}
&
- (c_{j+1}^{-1}\cdots c_{2g}^{-1}c_1\cdots c_{j-1})^{s-1}c_{j+1}^{-1}\cdots c_{i-1}^{-1}
+ c_j(c_{j-1}\cdots c_1c_{2g}^{-1}\cdots c_{j+1}^{-1})^{s-1}c_j^{-1}\cdots c_{i-1}^{-1},
\end{array}
$$
which is reduced to $0$ by $\myd{1,j,s-1}$ if $s>1$ or by $\myd{7,j}$ if $s=1$.

(C17) The composition set of $\myd{1,j,s}$ and $\myd{7,i}$ is non-empty only if $j=i$. In this situation, the unique composition is:
$$\myd{1,j,s}c_j - c_j(c_{j-1}\cdots c_1c_{2g}^{-1}\cdots c_{j+1}^{-1})^{s} \myd{7,j}
 =
 -(c_{j+1}^{-1}\cdots c_{2g}^{-1}c_1\cdots c_{j-1})^sc_j + c_j(c_{j-1}\cdots c_1c_{2g}^{-1}\cdots c_{j+1}^{-1})^{s},
  $$
which is reduced to $0$ (in $s$ steps) by $\myd{5,j+1}$ if $j<2g$ or by $\myd{4}$ if $j=2g$.




(C23) The composition set of $\myd{2,j,s}$ and $\myd{3}$ is non-empty only if $j=2g$. When $j=2g$, the unique composition is:
$$
\begin{array}{cl}
& \myd{2,2g,s} c_{2g-1}^{-1}\cdots c_1^{-1}- c_{2g}(c_1^{-1}\cdots c_{2g-1}^{-1})^s\myd{3} \\
= &
- (c_{2g-1}^{-1}\cdots c_1^{-1})^{s+1} + c_{2g}(c_1^{-1}\cdots c_{2g-1}^{-1})^{s+1} c_{2g}^{-1} \\
\rdto{\myd{2,2g,s+1}} & 0 .
\end{array}
$$


(C25) The composition set of $\myd{2,j,s}$ and $\myd{5,i}$ is non-empty only if $j\ge i$. When $j\ge i$, the unique composition is:
$$
\begin{array}{cl}
& \myd{2,j,s} c_{j+1}^{-1}\cdots c_{2g}^{-1}c_1\cdots c_{i-1}
 -
 c_j(c_{j+1}\cdots c_{2g}c_1^{-1}\cdots c_{j-1}^{-1})^{s-1} c_{j+1}\cdots c_{2g}c_1^{-1}\cdots  c_{i-1}^{-1}\myd{5,i} \\
= &
- (c_{j-1}^{-1}\cdots c_1^{-1}c_{2g}\cdots c_{j+1})^{s} c_{j+1}^{-1}\cdots c_{2g}^{-1}c_1\cdots c_{i-1} \\
  &
   \ + c_j(c_{j+1}\cdots c_{2g}c_1^{-1}\cdots c_{j-1}^{-1})^{s-1} c_{j+1}\cdots c_{2g}c_1^{-1}\cdots  c_{i-1}^{-1} c_{i-1}\cdots c_1 c_{2g}^{-1}\cdots c_i^{-1} \\
\rdtop{\myd{7}\cup \myd{8}} &
- (c_{j-1}^{-1}\cdots c_1^{-1}c_{2g}\cdots c_{j+1})^{s-1}c_{j-1}^{-1}\cdots c_{i+1}^{-1}c_i^{-1}\\
 &
\ + c_j(c_{j+1}\cdots c_{2g}c_1^{-1}\cdots c_{j-1}^{-1})^{s-1}c_{j}^{-1}c_{j-1}^{-1}\cdots c_{i+1}^{-1}c_i^{-1},
\end{array}
$$
which is reduced to $0$ by $\myd{2,j,s-1}$ if $s>1$ or by $\myd{8,j}$ if $s=1$.

(C27) The composition set of $\myd{2,j,s}$ and $\myd{7,i}$ is non-empty only if $j=i$. When $j=i$, the unique composition is:
$$ \myd{2,j,s} c_j - c_j(c_{j+1}\cdots c_{2g}c_1^{-1}\cdots c_{j-1}^{-1})^{s} \myd{7,j}
 =
 -(c_{j-1}^{-1}\cdots c_1^{-1}c_{2g}\cdots c_{j+1})^{s}c_j + c_j(c_{j+1}\cdots c_{2g}c_1^{-1}\cdots c_{j-1}^{-1})^{s},
  $$
which is reduced to $0$ (in $s$ steps) by $\myd{6,j}$.


(C36) If $2<j\le 2g$, the composition set $\comp(\myd{3}, \myd{6,j})$ is empty. If $j=2$, the unique composition is:
$$\myd{3}c_{2g}\cdots c_2  -  c_{2g}^{-1}\cdots c_2^{-1}\myd{6,2}
 =
 - c_1^{-1}\cdots c_{2g}^{-1}c_{2g}\cdots c_2 + c_{2g}^{-1}\cdots c_2^{-1}  c_2\cdots c_{2g}c_1^{-1}
 \rdto{\myd{7}} 0.
 $$

(C37) If $1<j\le 2g$, the composition set $\comp(\myd{3}, \myd{7,j})$ is empty. If $j=1$, the unique composition is:
$$\myd{3} c_1  -  c_{2g}^{-1}\cdots c_2^{-1}\myd{7,1}
 =
 - c_1^{-1}\cdots c_{2g}^{-1}c_1 + c_{2g}^{-1}\cdots c_2^{-1}
 \rdto{\myd{5,2}}
 - c_1^{-1}c_1c_{2g}^{-1}\cdots c_2^{-1} + c_{2g}^{-1}\cdots c_2^{-1}
 \rdto{\myd{7,1}}
  0.
 $$

(C41) If $2\le j <2g$, the composition set $\comp(\myd{4}, \myd{1,j,s})$ is empty. If $j=2g$, the unique composition is:
$$
\begin{array}{cl}
&
\myd{4}(c_{2g-1}\cdots c_1)^s c_{2g}^{-1} - c_1\cdots c_{2g-1}\myd{1,2g,s}\\
 = &
 - c_{2g}(c_{2g-1}\cdots c_1)^{s+1} c_{2g}^{-1} + (c_1\cdots c_{2g-1})^{s+1}\\
 \rdto{\myd{1,2g,s+1}} & 0.
\end{array}
$$

(C42) If $2\le j <2g$, the composition set $\comp(\myd{4}, \myd{2,j,s})$ is empty. If $j=2g$, the unique composition is:
$$
\begin{array}{cl}
&
\myd{4}(c_{1}^{-1}\cdots c_{2g-1}^{-1})^s c_{2g}^{-1} - c_1\cdots c_{2g-1}\myd{2,2g,s}\\
 = &
 - c_{2g}\cdots c_1(c_{1}^{-1}\cdots c_{2g-1}^{-1})^s c_{2g}^{-1} + c_1\cdots c_{2g-1}( c_{2g-1}^{-1}\cdots c_1^{-1})^{s}\\
 \rdtop{\myd{8}} &
 - c_{2g}(c_{1}^{-1}\cdots c_{2g-1}^{-1})^{s-1} c_{2g}^{-1} + ( c_{2g-1}^{-1}\cdots c_1^{-1})^{s-1},
\end{array}
$$
which is reduced to $0$ by  $\myd{2,2g,s-1}$ when $s>1$ or by $\myd{8,2g}$ when $s=1$.

(C48) If $j< 2g$, the composition set $\comp(\myd{4}, \myd{8,j})$ is empty. If $j=2g$, the unique composition is:
$$\myd{4} c_{2g}^{-1}  -  c_1\cdots c_{2g-1}\myd{8,2g}
 =
 - c_{2g}\cdots c_1c_{2g}^{-1} + c_1\cdots c_{2g-1}
 \rdto{\myd{1,2g,1}}
  0.
 $$

(C51) If $j\ne i-1$, the composition set $\comp(\myd{5,i}, \myd{1,j, s})$ is empty. If $j=i-1$, the unique composition is:
$$
\begin{array}{cl}
  &
\myd{5,j+1} (c_{j-1}\cdots c_1c_{2g}^{-1}\cdots c_{j+1}^{-1})^sc_j^{-1}-
 c_{j+1}^{-1}\cdots c_{2g}^{-1}c_1\cdots c_{j-1}\myd{1,j, s} \\
 = &
 - c_{j}(c_{j-1}\cdots c_1c_{2g}^{-1}\cdots c_{j+1}^{-1})^{s+1}c_j^{-1}
 +(c_{j+1}^{-1}\cdots c_{2g}^{-1}c_1\cdots c_{j-1})^{s+1}\\
\rdto{\myd{1,j,s+1}} & 0.
\end{array}
$$

(C52) If $i-1<j$, the composition set $\comp(\myd{5,i}, \myd{1,j,s})$ is empty. If $i-1\ge j$, the unique composition is:
$$
\begin{array}{cl}
  &
\myd{5,i} c_{i}\cdots c_{2g}c_1^{-1}\cdots c_{j-1}^{-1}(c_{i}\cdots c_{2g}c_1^{-1}\cdots c_{j-1}^{-1})^{s-1}c_j^{-1}-
 c_{i}^{-1}\cdots c_{2g}^{-1}c_1\cdots c_{j-1}\myd{1,j, s} \\
 = &
 - c_{i-1}\cdots c_j \cdots c_1c_{2g}^{-1}\cdots c_{i}^{-1} c_i\cdots c_{2g}c_1^{-1}\cdots c_{j-1}^{-1}
  (c_{i}\cdots c_{2g}c_1^{-1}\cdots c_{j-1}^{-1})^{s-1}c_j^{-1} \\
&
 \ + c_{i}^{-1}\cdots c_{2g}^{-1}c_1\cdots c_{j-1} (c_{j-1}^{-1}\cdots c_1^{-1}c_{2g}\cdots c_{j+1})^s\\
\rdto{\myd{7}\cup \myd{8}} &
 - c_{i-1}\cdots c_{j+1}c_j(c_{i}\cdots c_{2g}c_1^{-1}\cdots c_{j-1}^{-1})^{s-1}c_j^{-1}
 + c_{i-1}\cdots c_{j+1}(c_{j-1}^{-1}\cdots c_1^{-1}c_{2g}\cdots c_{j+1})^{s-1},
\end{array}
$$
which is reduced to $0$ by  $\myd{2,j,s-1}$ when $s>1$ or by $\myd{8,j}$ when $s=1$.

(C58) If $i-1\ne j$, the composition set $\comp(\myd{5,i}, \myd{8,j})$ is empty. If $i-1=j$, the unique composition is:
$$\myd{5,i} c_{i-1}^{-1}  -  c_i^{-1}\cdots c_{2g}^{-1} c_1 \cdots c_{i-2}\myd{8,i-1}
 =
 - c_{i-1}\cdots c_1c_{2g}^{-1}\cdots c_{i-1}^{-1}
 +  c_i^{-1}\cdots c_{2g}^{-1} c_1 \cdots c_{i-2},
 $$
which is reduced to $0$ by  $\myd{1,i-1,1}$ if $i>2$ or by $\myd{3}\cup \myd{8,1}$ if $i=2$.

(C61) If $i>j$, the composition set $\comp(\myd{6,i}, \myd{1,j,s})$ is empty. If $i\le j$, the unique composition is:
$$\begin{array}{cl}
&
\myd{6,i}c_{i-1}\cdots c_1c_{2g}^{-1}\cdots c_{j+1}^{-1}
 (c_{j-1}\cdots c_1c_{2g}^{-1}\cdots c_{j+1}^{-1})^{s-1} c_{j}^{-1}
 - c_{i-1}^{-1}\cdots c_1^{-1}c_{2g}\cdots c_{j+1} \myd{1,j,s}\\
 = &
 -  c_{i}\cdots c_{2g}c_1^{-1}\cdots c_{i-1}^{-1}
 c_{i-1}\cdots c_1c_{2g}^{-1}\cdots c_{j+1}^{-1}
 (c_{j-1}\cdots c_1c_{2g}^{-1}\cdots c_{j+1}^{-1})^{s-1} c_{j}^{-1} \\
 &
 \ + c_{i-1}^{-1}\cdots c_1^{-1}c_{2g}\cdots c_{j+1}
  (c_{j+1}^{-1}\cdots c_{2g}^{-1}c_1\cdots c_{j-1})^s\\
\rdtop{\myd{7}\cup \myd{8}} &
-c_ic_{i+1}\cdots c_{j-1}c_j(c_{j-1}\cdots c_1c_{2g}^{-1}\cdots c_{j+1}^{-1})^{s-1} c_{j}^{-1}
 + c_ic_{i+1}\cdots c_{j-1}(c_{j+1}^{-1}\cdots c_{2g}^{-1}c_1\cdots c_{j-1})^{s-1},
\end{array}
$$
which is reduced to $0$ by  $\myd{1,j,s-1}$ if $s>1$ or by $\myd{8,j}$ if $s=1$.

(C62) If $j\ne i$, the composition set $\comp(\myd{6,j}, \myd{2,i,s})$ is empty. If $j = i$, the unique composition is:
$$\begin{array}{cl}
&
\myd{6,j}
 (c_{j+1}\cdots c_{2g}c_1^{-1}\cdots c_{j-1}^{-1})^{s} c_j^{-1}
 - c_{j-1}^{-1}\cdots c_1^{-1}c_{2g}\cdots c_{j+1} \myd{2,i,s}\\
 = &
 -  c_{j}(c_{j+1}\cdots c_{2g}c_1^{-1}\cdots c_{j-1}^{-1})^{s+1}c_j^{-1}
  + (c_{j-1}^{-1}\cdots c_1^{-1}c_{2g}\cdots c_{j+1})^{s+1} \\
\rdto{\myd{2,j,s+1}} & 0.
\end{array}
$$

(C68) If $j\ne i$, the composition set $\comp(\myd{6,j}, \myd{8,i})$ is empty. If $j = i$, the unique composition is:
$$\myd{6,j}c_j^{-1} -  c_{j-1}^{-1}\cdots c_1^{-1}c_{2g}\cdots c_{j+1}\myd{8,i}
 = -c_{j}\cdots c_1c_{2g}^{-1}\cdots c_{j}^{-1}
 + c_{j-1}^{-1}\cdots c_1^{-1}c_{2g}\cdots c_{j+1}
\rdto{\myd{1,j,1}} 0.
$$

(C71) If $j\ne i$, the composition set $\comp(\myd{7,j}, \myd{1,i,s})$ is empty. If $j = i$, the unique composition is:
$$\begin{array}{cl}
&
\myd{7,j}(c_{j-1}\cdots c_1c_{2g}^{-1}\cdots c_{j+1}^{-1})^s  c_{j}^{-1} - c_j^{-1} \myd{1,j,s}\\
 = &
 - (c_{j-1}\cdots c_1c_{2g}^{-1}\cdots c_{j+1}^{-1})^s  c_{j}^{-1}
   + c_j^{-1}(c_{j+1}^{-1}\cdots c_{2g}^{-1}c_1\cdots c_{j-1})^s \\
\rdtop{\myd{5,j}}  & 0.
\end{array}
$$

(C72) If $j\ne i$, the composition set $\comp(\myd{7,j}, \myd{2,i,s})$ is empty. If $j = i$, the unique composition is:
$$\begin{array}{cl}
&
\myd{7,j}(c_{j+1}\cdots c_{2g}c_1^{-1}\cdots c_{j-1}^{-1})^s  c_{j}^{-1} - c_j^{-1} \myd{2,j,s}\\
 = &
 - (c_{j+1}\cdots c_{2g}c_1^{-1}\cdots c_{j-1}^{-1})^s  c_{j}^{-1}
   + c_j^{-1}(c_{j-1}^{-1}\cdots c_1^{-1}c_{2g}\cdots c_{j+1})^s,
\end{array}
$$
which is reduced (in $s$ steps) to $0$ by  $\myd{6,j+1}$ if $j<2g$ or by $\myd{3}$ if $j=2g$.

(C74) If $j\ne 1$, the composition set $\comp(\myd{7,j}, \myd{4})$ is empty. If $j = 1$, the unique composition is:
$$\begin{array}{rcl}
 \myd{7,1}c_2\cdots c_{2g} - c_1^{-1} \myd{4}
 &  = &  - c_2\cdots c_{2g} + c_1^{-1} c_{2g}\cdots c_1\\
 &\rdto{\myd{6,2}} & - c_2\cdots c_{2g} + c_2\cdots c_{2g}c_1^{-1}c_1\\
 &\rdto{\myd{7,1}} & 0.
\end{array}$$

(C83) If $j\ne 2g$, the composition set $\comp(\myd{8,j}, \myd{3})$ is empty. If $j = 2g$, the unique composition is:
$$
 \myd{8,2g}c_{2g-1}^{-1}\cdots c_1^{-1} - c_{2g}\myd{3}
   =   - c_{2g-1}^{-1}\cdots c_1^{-1} + c_{2g}c_1^{-1} \cdots c_{2g}^{-1}
 \rdto{\myd{2,2g,1}}  0.
$$

(C85) If $j\ne i$, the composition set $\comp(\myd{8,j}, \myd{5,i})$ is empty. If $j = i$, the unique composition is:
$$
 \myd{8,j}c_{j+1}^{-1}\cdots c_{2g-1}^{-1}c_1\cdots c_{j-1} - c_j\myd{5,j}
 =
 - c_{j+1}^{-1}\cdots c_{2g-1}^{-1}c_1\cdots c_{j-1}
  + c_j \cdots c_1c_{2g}^{-1}\cdots c_j^{-1}
 \rdto{\myd{1,j,1}} 0.
$$

(C86) If $j\ne i-1$, the composition set $\comp(\myd{8,j}, \myd{6,i})$ is empty. If $j = i-1$, the unique composition is:
$$
\begin{array}{cl}
 & \myd{8,j}c_{j-1}^{-1}\cdots c_1^{-1}c_{2g}\cdots c_{j+1} - c_j\myd{6,j+1}\\
 = &
 - c_{j-1}^{-1}\cdots c_1^{-1}c_{2g}\cdots c_{j+1}
  + c_j c_{j+1}\cdots c_{2g}c_1^{-1}\cdots c_j^{-1} \\
 \rdto{\myd{2,j,1}} &0.
\end{array}
$$
Here, we also drop the proofs of two obvious cases: C(78) and C(87).
\end{proof}

Given any word $\gamma$ in the letter set $\{c_1, c^{-1}_1,\ldots c_{2g}, c^{-1}_{2g}\}$, the reduced form of $\gamma$ can be computed easily. Hence, this theorem can be regarded as a re-visiting of the word problem of the surface groups.

Note that in our \gr basis $D$, the inverse of a leading word of any polynomial in $D$ is always the leading word of another element. More precisely, $\lw(\myd{1,j,s})= \lw(\myd{2,j,s})^{-1}$, $\lw(\myd{3})= \lw(\myd{4})^{-1}$, $\lw(\myd{5,i})= \lw(\myd{6,i})^{-1}$, $\lw(\myd{7,k})= \lw(\myd{7,k})^{-1}$ and $\lw(\myd{8,k})= \lw(\myd{8,k})^{-1}$ for all possible $i,j,k$. Thus, we have

\begin{Corollary}\label{Cor:inverse}
Let $\alpha =a_1\cdots a_k$ be a word in the letter set $\{c_1, c^{-1}_1,\ldots c_{2g}, c^{-1}_{2g}\}$. Then $\alpha$ is $D$-reducible if and only if its inverse $a_k^{-1}\cdots a_1^{-1}$ is $D$-reducible.
\end{Corollary}

This property will play a central role in our treatment to geometric intersections, and it guarantees that the components of common value pairs is the same as the common value classes for loops determined by cyclically reduced words.

\section{Cyclically reduced words and common value classes}

By using cyclically $D$-reduced words, we shall show that common value classes of two piecewise-geodesic loops coincide with components of common value pairs.

\begin{Definition}
A word $\alpha = a_1\cdots a_s$ is said to be cyclically $D$-reduced if all of its rotations: $$a_1\cdots a_s,\  a_2a_3\cdots a_sa_1,\ a_3\cdots a_sa_1a_2,\  \ldots,\  a_sa_1\cdots a_{s-1}$$ are all $D$-reduced.
\end{Definition}

\begin{Corollary}
A word $\alpha$ is cyclically $D$-reduced if and only if $\alpha^2$ is $D$-reduced.
\end{Corollary}

Because of our \gr basis $D$, $\alpha$ is $D$-reduced if and only if any of its rotation does not contain any leading word in Theorem~\ref{grobner}. Thus, $D$-reducibility is algorithmically decidable.



It is clear that any two words will  present conjugate elements in $\pi_1(F_g)$ if they have the same cyclically $D$-reduced form. But the converse is not true. For example, $c_4c_1^{-1}c_3$ and $c_4c_3c_1^{-1}$ are both cyclically $D$-reduced, and hence are different cyclically $D$-reduced forms. But they are conjugate in $\pi_1(F_2)$, because $c_2^{-1}c_1^{-1}(c_4c_3c_1^{-1})c_1c_2 = c_3c_4c_1^{-1}c_2^{-1} c_2 = c_3c_4c_1^{-1} \in \pi_1(F_2)$.

The importance of ``cyclically $D$-reduced'' lies in:

\begin{thm}\label{thm-connect-class}
Let $\varphi,\psi: S^1\to F_g$ be two piecewise-geodesic loops which are determined by cyclically $D$-reduced words.
Then for any lifting $\tilde \varphi$ of $\varphi$ and any lifting $\tilde \psi$ of $\psi$, the set of common value pair of $\tilde \varphi$ and $\tilde \psi$ is connected, and hence each common value class is connected.
\end{thm}

\begin{proof}
Let $\mu = u_1\cdots u_m$ and $\nu = v_1\cdots v_n$ be two cyclically $D$-reduced words determining $\varphi$ and $\psi$, respectively. Suppose  that two liftings $\tilde \varphi, \tilde \psi: \mathbb{R}^1\to \mathbb{D}^2$ have two common value pairs: $(x'_\mu, x'_\nu)$ and $(x''_\mu, x''_\nu)$. By Theorem~\ref{th:cvp-component}, we can find common value pairs $(\frac{k'}{m}, \frac{l'}{n})$ and $(\frac{k''}{m}, \frac{l''}{n})$ with integers $k',l', k'', l''$ such that $(x'_\mu, x'_\nu)$ and $(\frac{k'}{m}, \frac{l'}{n})$ lie in the same component of common value pairs of $\tilde \varphi$ and  $\tilde \psi$, and so $(x''_\mu, x''_\nu)$ and $(\frac{k''}{m}, \frac{l''}{n})$.

Suppose that $\tilde \varphi_S$ and  $\tilde \psi_S$ are respectively standard liftings (see Definition~\ref{def-standard-lifting}) of $\varphi$ and $\psi$. Then $\tilde \varphi = \alpha\tilde \varphi_S$ and $\tilde \psi = \beta\tilde \psi_S$, where $\alpha, \beta\in \deck(\mathbb{D}^2)\cong \pi_1(F_g, y_0)$.
Since $\tilde \varphi(\frac{k'}{m}) = \tilde \psi(\frac{l'}{n})$, we have that $\alpha\tilde \varphi_S(\frac{k'}{m}) = \beta\tilde \psi_S(\frac{l'}{n})$. It follows that $\alpha u_1\cdots u_{k'}(\tilde y_0) = \beta v_1\cdots v_{l'}(\tilde y_0)$. By the uniqueness of covering transformation, we have that $\alpha u_1\cdots u_{k'} = \beta v_1\cdots v_{l'}$ in $\deck(\mathbb{D}^2)\cong \pi_1(F_g, y_0)$. With the same reason, we also have that $\alpha u_1\cdots u_{k''} = \beta v_1\cdots v_{l''}$. Thus, $$\alpha^{-1}\beta = u_1\cdots u_{k'} (v_1\cdots v_{l'})^{-1} = u_1\cdots u_{k''} (v_1\cdots v_{l''})^{-1}.$$
We obtain that
$$(u_1\cdots u_{k''})^{-1}u_1\cdots u_{k'}  =  (v_1\cdots v_{l''})^{-1} v_1\cdots v_{l'} \in \pi_1(F_g, y_0).$$
Hence, after some cancelation (or say $\myd{7}\sqcup\myd{8}$-reductions), as two words, we have:
$$(u_{\min\{k',k''\}+1}\cdots u_{\max\{k',k''\}})^{sgn(k'-k'')}=
(v_{\min\{l',l''\}+1}\cdots v_{\max\{l',l''\}})^{sgn(l'-l'')}.$$
Note that both sides are $D$-reduced because they are respectively subwords of $D$-reduced words $\mu=u_1\cdots u_m$ or $\mu^{-1}$ and $\nu= v_1\cdots v_n$ or $\nu^{-1}$. Note that $\mu^{-1}$ and $\nu^{-1}$ are both $D$-reduced from Corollary~\ref{Cor:inverse}.

If $k'=k''$, then the left hand side of equality above is trivial, and therefore the right hand side is also trivial, i.e. $l'=l''$. This implies that $(x'_\mu, x'_\nu)$ and $(x''_\mu, x''_\nu)$ lie in the same component of common value pairs, containing $(\frac{k'}{m}, \frac{l'}{n})=(\frac{k''}{m}, \frac{l''}{n})$.

If $k'>k''$, then the left hand side of equality above is $u_{k''+1}\cdots u_{k'}$. The uniqueness of $D$-reduced form implies that there are only two possibilities: (1) $l'-l''=k'-k''$ and $u_{k''+j} = v_{l''+j}$ for $j=1, \ldots, k'-k''$, (2) $l''-l'=k'-k''$ and $u_{k''+j} = v_{l'-j+1}^{-1}$ for $j=1,\ldots, k'-k''$. By Theorem~\ref{th:cvp-component} and its Corollary, we can see that $(\frac{k'}{m}, \frac{l'}{n})$ and $(\frac{k''}{m}, \frac{l''}{n})$ lie in the same component of common value pairs. It follows that $(x'_\mu, x'_\nu)$ and $(x''_\mu, x''_\nu)$ lie in the same component.

The proof of case $k'<k''$ is similar.
\end{proof}

This theorem implies that the number of  common value classes any two  piecewise-geodesic loops is practically computable.

\section{Indices of common value classes}

In this section, we shall show a method to compute the indices of common value classes. Of most importance is to determine if a common value class is essential, i.e. has a non-zero index.

Recall from \cite[Def. 3.1]{guying} that the (homology) homomorphism index $\mathcal{L}_*(\varphi\times \psi, C, \Delta)$ of an isolated common value subset $C$ of $\varphi, \psi\colon S^1\to F_g$ is defined to be the composition of following:
$$H_*(S^1\times S^1) \stackrel{j_*}{\to} H_*(S^1\times S^1, S^1\times S^1-C) \stackrel{e^{-1}_*}{\to} H_*(N, N-C)  \stackrel{(\varphi\times\psi)_*}{\to} H_*(F_g\times F_g, F_g\times F_g-\Delta),$$
where $N$ is a neighborhood of $C$ with $N\cap \cvp(\varphi, \psi)= C$. Clearly, the non-triviality of the homomorphisms happens at dimension $2$ only. Note that $H_2(S^1\times S^1)\cong H_2(F_g\times F_g, F_g\times F_g-\Delta)\cong \mathbb{Z}$.  The homomorphism index $\mathcal{L}_*(\varphi\times \psi, C, \Delta)$ can be converted into a numerical one if both of these two homology groups have chosen generators. This leads to the following:

\begin{Proposition}\label{Prop-essential-intersectionnumber}
Let $C$ be an isolated common value subset of two maps $\varphi, \psi\colon S^1\to F_g$. Then homomorphism index $\mathcal{L}_2(\varphi\times \psi, C, \Delta)$ of $\varphi\times \psi$ at $C$ is given by $[S^1]\times [S^1]\mapsto i(\varphi, \psi; C)[\tau_\Delta]$, where $[S^1]$ is the fundamental class, $i(\varphi, \psi; C)$ is the intersection number of $\varphi$ and $\psi$ at $C$, and $[\tau_\Delta]$ is the Thom class of diagonal $\Delta$ in $F_g\times F_g$.
\end{Proposition}

\begin{proof} It is obvious by definition of intersection number. Here, the orientation of the circle $S^1$ and the surface $F_g$ are given respectively by their natural coordinates of $\{e^{\theta i}\}$ and $\mathbb{D}^2$.
\end{proof}

\begin{Corollary}
Let $C$ be an isolated common value subset of two maps $\varphi, \psi\colon S^1\to F_g$. Then $C$ is essential, i. e. has a non-zero homomorphism index, if and only if the intersection number of $\varphi$ and $\psi$ at $C$ is not zero. Moreover, if $C$ is connected, then the intersection number of $\varphi$ and $\psi$ at $C$ is $-1$, $0$ or $1$.
\end{Corollary}

In this paper, we shall use the local intersection number to indicate the index of a set of common value pair.  It is known that if $\eta, \eta': S^1\to F_g$ are homotopic loops, then their homotopy related liftings $\tilde \eta, \tilde \eta'$ have the same ending points on the circle of infinity of $\mathbb{D}^2$, (cf. \cite[Lemma 2.3]{Casson}). Thus, the homotopy invariance of index of a common value class is almost obvious in our case: for maps from the circle to a surface.

Next theorem shows that the index of each component of common value pair can be read locally from subword pair determining this component, if the loops in consideration are piecewise-geodesic.

Let us fix some notations. For any three points $P, Q, R$ in the oriented circle $S^1$, the number $\Theta(P, Q, R)$ is defined to be $1$ if $P,Q,R$ are distinct points and the cyclic order of $P,Q,R$ coincides with the given orientation of $S^1$, to be $-1$ if $P,Q,R$ are distinct points and the cyclic order of $P,Q,R$ is different from the given orientation of $S^1$, and to be $0$ otherwise.

\newcommand{\myss}[2]{-{\mathrm T}^{#1}_{#2}}
\newcommand{\mytt}[2]{{\mathrm T}^{#1}_{#2}}
\newcommand{\mydd}[3]{\Theta(#1,#2,#3)}

\begin{Theorem}\label{thm-index}
Let $\varphi,\psi: S^1\to F_g$ be two piecewise-geodesic loops which are determined by cyclically $D$-reduced words $u_1\cdots u_m$ and $v_1\cdots v_n$, respectively. Each $u_i$ and $v_j$ lies in the letter set $\{c_1, c^{-1}_1,\ldots c_{2g}, c^{-1}_{2g}\}$. Then all possible components of common value set $\cvp(\varphi, \psi)$ and their indices (intersection numbers) are listed as follows.
\begin{center}
\begin{tabular}{|c|l|l|}
  \hline
  Type & data of $C$ & \ \ \ $ind(\varphi, \psi; C)$ \\ \hline
  (1) & $(k,l,0)$ &
  $\frac{1}{2}\big(\Theta(-T(u_{k}), -T(v_{l}), T(u_{k+1}))
  + \Theta(T(u_{k+1}), T(v_{l+1}), -T(u_{k}))\big)$
 \\ \hline
  (2) & $(k,l,q)$
  & $\frac{1}{2}\big(\Theta(-T(u_{k}), -T(v_l), T(u_{k+1})) + \Theta(T(u_{k+q+1}), T(v_{l+q+1}), -T(u_{k+q}))\big) $ \\ \hline
  (3) & $(k,l,-q)$
 & $-\frac{1}{2}\big(\Theta(-T(u_{k}),  T(v_{l+q+1}), T(u_{k+1})) - \Theta(T(u_{k+q+1}), -T(v_{l}), -T(u_{k+q}))\big) $ \\ \hline
  (4)& &\ \ $0$ \\ \hline
  (5) & & \ \ $0$ \\   \hline
\end{tabular}
\end{center}
Here $q$ is a positive integer.
\end{Theorem}

\begin{proof}
All possible components of $\cvp(\varphi, \psi)$ are already given in Theorem~\ref{th:cvp-component}. The notation $T(\cdot)$ is defined in (\ref{eq-T}).

Consider the first type in Theorem~\ref{th:cvp-component}: the component $C$ of the set $\cvp(\varphi, \psi)$ is a singleton $\{( e^{\frac{2\pi k i}{m}}, e^{\frac{2\pi l i}{n}})\}$. Then, as in the proof of Theorem~\ref{th:cvp-component}, we have that
$$\{(e^{\frac{2\pi k i}{m}},  e^{\frac{2\pi l i}{n}})\} = (p_{S^1}\times p_{S^1})(\cvp((u_1\cdots u_k)^{-1}\tilde \varphi_S,\ (v_1\cdots v_l)^{-1}\tilde \psi_S)), $$
where $\tilde \varphi_S$ and $\tilde \psi_S$ are respectively standard liftings of $\varphi$ and $\psi$. By item (1) of the Corollary of Theorem~\ref{th:cvp-component} and our assumption, the point $\tilde y_0$ is the unique intersection of $(u_1\cdots u_k)^{-1}\tilde \varphi_S$ and $(v_1\cdots v_l)^{-1}\tilde \psi_S$. By definition of piecewise-geodesic loops and standard liftings, we have $\tilde \varphi_S(\frac{s}{m}) = u_1\cdots u_s(\tilde y_0)$ and $\tilde \psi_S(\frac{t}{n}) = v_1\cdots v_t(\tilde y_0)$ for any integers $s$ and $t$. Thus, around $\tilde y_0$, we can see the liftings $\tilde u_k, \tilde u_{k+1}, \tilde v_l,\tilde v_{l+1}$ of $u_k, u_{k+1}, v_l, v_{l+1}$. Here, $u_s$'s and $v_t$'s are considered as sub-loops of $\varphi$ and $\psi$, respectively. Corresponding liftings are sub-arcs of  $\tilde \varphi_S$ or $\tilde \psi_S$.

Thus, we can obtain the local intersection number of $(u_1\cdots u_k)^{-1}\tilde \varphi_S$ and $(v_1\cdots v_l)^{-1}\tilde \psi_S$ at $\tilde y_0$, by using the positions of attracting and expanding fixed points $T(\cdot)$'s (defined in (\ref{eq-T})) of corresponding generators. The local intersections and their indices are illustrated as follows.
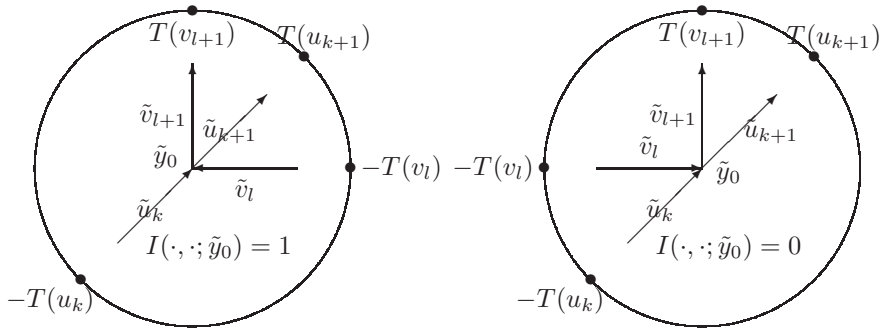
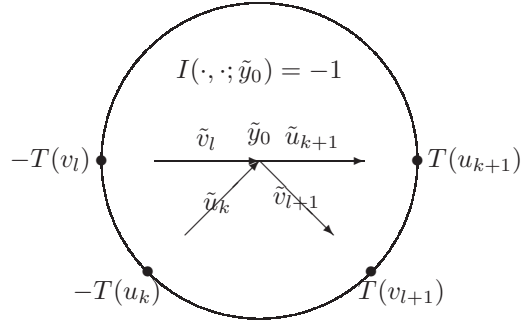
\begin{center}
\setlength{\unitlength}{0.7mm}
\begin{picture}(60,60)(-30,-30)
\mycl
\put(-14.142,-14.1420){\vector(1,1){14.142}}
\put(0,0){\vector(1,1){14.1420}}
\put(20,0){\vector(-1,0){20}}
\put(0,0){\vector(0,1){20}}

\put(-5,0){\makebox(0,0)[cb]{$\tilde y_0$}}
\put(-8,-8){\makebox(0,0)[cc]{$\tilde u_k$}}
\put(7.1,7.1){\makebox(0,0)[cc]{$\tilde u_{k+1}$}}
\put(10,-2){\makebox(0,0)[ct]{$\tilde v_l$}}
\put(-1,10){\makebox(0,0)[rc]{$\tilde v_{l+1}$}}

\put(21.2132,21.2132){\makebox(0,0)[cc]{$\bullet$}}
\put(25,25){\makebox(0,0)[cc]{$T(u_{k+1})$}}
\put(00,30){\makebox(0,0)[cc]{$\bullet$}}
\put(00,28){\makebox(0,0)[ct]{$T(v_{l+1})$}}
\put(30,0){\makebox(0,0)[cc]{$\bullet$}}
\put(32,0){\makebox(0,0)[lc]{$-T(v_{l})$}}
\put(-21.2132,-21.2132){\makebox(0,0)[cc]{$\bullet$}}
\put(-27,-25){\makebox(0,0)[cc]{$-T(u_k)$}}

\put(5,-15){\makebox(0,0)[cc]{$I(\cdot, \cdot;
\ C)=1$}}
\end{picture}
\hspace{4cm}
\begin{picture}(60,60)(-30,-30)
\mycl
\put(-14.142,-14.1420){\vector(1,1){14.142}}
\put(0,0){\vector(1,1){14.1420}}
\put(-20,0){\vector(1,0){20}}
\put(0,0){\vector(0,1){20}}

\put(5,-3){\makebox(0,0)[cb]{$\tilde y_0$}}
\put(-8,-8){\makebox(0,0)[cc]{$\tilde u_k$}}
\put(13,7.1){\makebox(0,0)[cc]{$\tilde u_{k+1}$}}
\put(-10,2){\makebox(0,0)[cb]{$\tilde v_l$}}
\put(-1,10){\makebox(0,0)[rc]{$\tilde v_{l+1}$}}

\put(21.2132,21.2132){\makebox(0,0)[cc]{$\bullet$}}
\put(25,25){\makebox(0,0)[cc]{$T(u_{k+1})$}}
\put(00,30){\makebox(0,0)[cc]{$\bullet$}}
\put(00,28){\makebox(0,0)[ct]{$T(v_{l+1})$}}
\put(-30,0){\makebox(0,0)[cc]{$\bullet$}}
\put(-32,0){\makebox(0,0)[rc]{$-T(v_{l})$}}
\put(-21.2132,-21.2132){\makebox(0,0)[cc]{$\bullet$}}
\put(-27,-25){\makebox(0,0)[cc]{$-T(u_k)$}}

\put(5,-15){\makebox(0,0)[cc]{$I(\cdot, \cdot;\ C)=0$}}
\end{picture}
\end{center}

\begin{center}
\setlength{\unitlength}{0.7mm}
\begin{picture}(60,60)(-30,-30)
\mycl
\put(-14.142,-14.1420){\vector(1,1){14.142}}
\put(0,0){\vector(1,0){20}}
\put(-20,0){\vector(1,0){20}}
\put(0,0){\vector(1,-1){14.142}}

\put(0,3){\makebox(0,0)[cb]{$\tilde y_0$}}
\put(-8,-8){\makebox(0,0)[cc]{$\tilde u_k$}}
\put(10,2){\makebox(0,0)[cb]{$\tilde u_{k+1}$}}
\put(7.1,-7.1){\makebox(0,0)[cc]{$\tilde v_{l+1}$}}
\put(-10,2){\makebox(0,0)[cb]{$\tilde v_l$}}

\put(30,0){\makebox(0,0)[cc]{$\bullet$}}
\put(32,0){\makebox(0,0)[lc]{$T(u_{k+1})$}}
\put(21.2132,-21.2132){\makebox(0,0)[cc]{$\bullet$}}
\put(27,-25){\makebox(0,0)[cc]{$T(v_{l+1})$}}
\put(-30,0){\makebox(0,0)[cc]{$\bullet$}}
\put(-32,0){\makebox(0,0)[rc]{$-T(v_{l})$}}
\put(-21.2132,-21.2132){\makebox(0,0)[cc]{$\bullet$}}
\put(-27,-25){\makebox(0,0)[cc]{$-T(u_k)$}}

\put(0,17){\makebox(0,0)[cc]{$I(\cdot, \cdot; \ C)=-1$}}
\end{picture}
\end{center}
The proof for the components of type (2) and (3) are similar. The following figure shows one case.
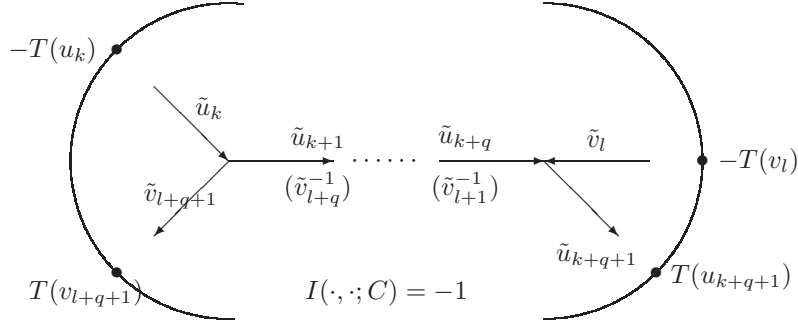
\begin{center}
\setlength{\unitlength}{0.7mm}
\begin{picture}(120,60)(-30,-30)
\qbezier(2.8,29.87)(1.23,30.02)(-0.36,30.)
\qbezier(-0.36,30.)(-1.94,29.98)(-3.51,29.79)
\qbezier(-3.51,29.79)(-5.09,29.61)(-6.63,29.26)
\qbezier(-6.63,29.26)(-8.18,28.91)(-9.68,28.4)
\qbezier(-9.68,28.4)(-11.18,27.89)(-12.61,27.22)
\qbezier(-12.61,27.22)(-14.05,26.55)(-15.41,25.74)
\qbezier(-15.41,25.74)(-16.77,24.93)(-18.04,23.97)
\qbezier(-18.04,23.97)(-19.3,23.02)(-20.46,21.94)
\qbezier(-20.46,21.94)(-21.62,20.86)(-22.66,19.66)
\qbezier(-22.66,19.66)(-23.7,18.47)(-24.6,17.17)
\qbezier(-24.6,17.17)(-25.51,15.87)(-26.27,14.48)
\qbezier(-26.27,14.48)(-27.04,13.09)(-27.65,11.63)
\qbezier(-27.65,11.63)(-28.27,10.17)(-28.72,8.66)
\qbezier(-28.72,8.66)(-29.18,7.14)(-29.48,5.58)
\qbezier(-29.48,5.58)(-29.77,4.03)(-29.9,2.45)
\qbezier(-29.9,2.45)(-30.03,0.87)(-29.99,-0.72)
\qbezier(-29.99,-0.72)(-29.95,-2.3)(-29.75,-3.87)
\qbezier(-29.75,-3.87)(-29.54,-5.44)(-29.18,-6.98)
\qbezier(-29.18,-6.98)(-28.81,-8.52)(-28.28,-10.02)
\qbezier(-28.28,-10.02)(-27.75,-11.51)(-27.07,-12.94)
\qbezier(-27.07,-12.94)(-26.38,-14.37)(-25.55,-15.72)
\qbezier(-25.55,-15.72)(-24.72,-17.07)(-23.76,-18.32)
\qbezier(-23.76,-18.32)(-22.79,-19.58)(-21.69,-20.72)
\qbezier(-21.69,-20.72)(-20.6,-21.87)(-19.39,-22.89)
\qbezier(-19.39,-22.89)(-18.18,-23.91)(-16.87,-24.81)
\qbezier(-16.87,-24.81)(-15.56,-25.7)(-14.17,-26.44)
\qbezier(-14.17,-26.44)(-12.77,-27.19)(-11.3,-27.79)
\qbezier(-11.3,-27.79)(-9.84,-28.39)(-8.31,-28.83)
\qbezier(-8.31,-28.83)(-6.79,-29.26)(-5.23,-29.54)
\qbezier(-5.23,-29.54)(-3.67,-29.82)(-2.09,-29.93)
\qbezier(-2.09,-29.93)(-0.51,-30.04)(1.07,-29.98)
\qbezier(60.,-30.)(61.58,-30.)(63.16,-29.83)
\qbezier(63.16,-29.83)(64.73,-29.67)(66.28,-29.33)
\qbezier(66.28,-29.33)(67.83,-29.)(69.34,-28.51)
\qbezier(69.34,-28.51)(70.84,-28.02)(72.29,-27.37)
\qbezier(72.29,-27.37)(73.73,-26.72)(75.1,-25.92)
\qbezier(75.1,-25.92)(76.47,-25.12)(77.75,-24.19)
\qbezier(77.75,-24.19)(79.03,-23.25)(80.2,-22.18)
\qbezier(80.2,-22.18)(81.37,-21.12)(82.42,-19.93)
\qbezier(82.42,-19.93)(83.47,-18.75)(84.4,-17.46)
\qbezier(84.4,-17.46)(85.32,-16.17)(86.1,-14.79)
\qbezier(86.1,-14.79)(86.88,-13.41)(87.51,-11.96)
\qbezier(87.51,-11.96)(88.14,-10.51)(88.62,-9.)
\qbezier(88.62,-9.)(89.09,-7.49)(89.41,-5.93)
\qbezier(89.41,-5.93)(89.72,-4.38)(89.87,-2.8)
\qbezier(89.87,-2.8)(90.02,-1.23)(90.,0.36)
\qbezier(90.,0.36)(89.98,1.94)(89.79,3.51)
\qbezier(89.79,3.51)(89.61,5.09)(89.26,6.63)
\qbezier(89.26,6.63)(88.91,8.18)(88.4,9.68)
\qbezier(88.4,9.68)(87.89,11.18)(87.22,12.61)
\qbezier(87.22,12.61)(86.55,14.05)(85.74,15.41)
\qbezier(85.74,15.41)(84.93,16.77)(83.97,18.04)
\qbezier(83.97,18.04)(83.02,19.3)(81.94,20.46)
\qbezier(81.94,20.46)(80.86,21.62)(79.66,22.66)
\qbezier(79.66,22.66)(78.47,23.7)(77.17,24.6)
\qbezier(77.17,24.6)(75.87,25.51)(74.48,26.27)
\qbezier(74.48,26.27)(73.09,27.04)(71.63,27.65)
\qbezier(71.63,27.65)(70.17,28.27)(68.66,28.72)
\qbezier(68.66,28.72)(67.14,29.18)(65.58,29.48)
\qbezier(65.58,29.48)(64.03,29.77)(62.45,29.9)
\qbezier(62.45,29.9)(61.23,30.)(60.,30.)

\put(-14.142, 14.1420){\vector(1,-1){14.142}}
\put(0,0){\vector(-1,-1){14.142}}
\put(80,0){\vector(-1,0){20}}
\put(60,0){\vector(1,-1){14.142}}

\put(45,2){\makebox(0,0)[cb]{$\tilde u_{k+q}$}}
\put(45,-2){\makebox(0,0)[ct]{($\tilde v_{l+1}^{-1}$)}}
\put(40,0){\vector(1,0){20}}

\put(17,2){\makebox(0,0)[cb]{$\tilde u_{k+1}$}}
\put(17,-2){\makebox(0,0)[ct]{($\tilde v_{l+q}^{-1}$)}}
\put(0,0){\vector(1,0){20}}

\put(30,0){\makebox(0,0)[cc]{$\cdots\cdots$}}

\put(70,-18){\makebox(0,0)[cc]{$\tilde u_{k+q+1}$}}
\put(81.2132,-21.2132){\makebox(0,0)[cc]{$\bullet$}}
\put(84,-22){\makebox(0,0)[lc]{$T(u_{k+q+1})$}}

\put(70,2){\makebox(0,0)[cb]{$\tilde v_{l}$}}
\put(90,0){\makebox(0,0)[cc]{$\bullet$}}
\put(93,0){\makebox(0,0)[lc]{$-T(v_{l})$}}

\put(-9,-9){\makebox(0,0)[cb]{$\tilde v_{l+q+1}$}}
\put(-21.2132,-21.2132){\makebox(0,0)[cc]{$\bullet$}}
\put(-27,-25){\makebox(0,0)[cc]{$T(v_{l+q+1})$}}

\put(-1,10){\makebox(0,0)[rc]{$\tilde u_{k}$}}
\put(-21.2132,21.2132){\makebox(0,0)[cc]{$\bullet$}}
\put(-25,21){\makebox(0,0)[rc]{$-T(u_{k})$}}

\put(30,-25){\makebox(0,0)[cc]{$I(\cdot, \cdot; C)=-1$}}
\end{picture}
\end{center}

For the component of type (4), corresponding intersection set is the whole loop. Thus, $\myim(\varphi)=\myim(\psi)$. We can push $\varphi$ a little along its  normal direction into $\varphi'$.  We obtain that $\myim(\varphi')\cap\myim(\psi) = \emptyset$, especially the component in consideration is moved out. The homotopy invariance of index implies that such component has index $0$. The proof of type (5) is the same.
\end{proof}

By Theorem~\ref{thm-connect-class}, each common value class of $\varphi$ and $\psi$ is just a component of common value pair if $\varphi$ and $\psi$ are geodesic loops determined by cyclically $D$-reduced words. Our theorem~\ref{thm-index} shows that the number of essential common value classes is computable. Moreover, such a computation is really symbolic one, because we have

\begin{Proposition}
Let $w_1, w_2, w_3$ be three letters in $\{c_1, c^{-1}_1,\ldots c_{2g}, c^{-1}_{2g}\}$, and let $\epsilon_1, \epsilon_2, \epsilon_3 =\pm 1$. If $w_1^{\epsilon_1}, w_2^{\epsilon_2}, w_3^{\epsilon_3}$ are distinct, then
$$\Theta(\epsilon_1 T(w_1), \epsilon_2 T(w_2), \epsilon_3 T(w_3)) = \mathrm{sgn}(\vartheta_g(w_1^{\epsilon_1}), \vartheta_g(w_2^{\epsilon_2}), \vartheta_g(w_3^{\epsilon_3})),$$
 where $\vartheta_g: \{c_1, c^{-1}_1,\ldots c_{2g}, c^{-1}_{2g}\}$ is a one-to-one correspondence given by $$c_j\mapsto j-1+(1-(-1)^j)g,\ \ c_j^{-1}\mapsto j-1+(1+(-1)^j)g.$$
\end{Proposition}

Thus, we can compute  locally the indices of all common value classes, instead of  comparing of translation axis used in \cite[Sec. 6]{Reinhart1962}, which was a very hard job if a word is large.

Using the data of components of the set of common value pairs, we obtain immediately

\begin{Proposition}\label{prop-s-time}
Let $\varphi,\psi, \varphi'$ and $\psi'$ be loops which determined by cyclically $D$-reduced words $u_1\cdots u_m$, $v_1\cdots v_n$, $(u_1\cdots u_m)^s$ and  $(v_1\cdots v_n)^t$, respectively, where $s$ and $t$ are positive integers. Then the number of essential common value classes of $\varphi'$ and $\psi'$ is $s\times t$ times of the number of essential common value classes of $\varphi$ and $\psi$.
\end{Proposition}

\begin{proof}
By Theorem~\ref{thm-index}, the components of type (4) and (5) in Theorem~\ref{th:cvp-component} have indices zero. Thus, all essential common value classes are components of the type (1), (2) or (3). It is obvious that each component of common value pairs of $\varphi$ and $\psi$ in one of these three forms gives $s\times t$ components of common value pairs of $\varphi'$ and $\psi'$. Hence, we are done.
\end{proof}

The number of essential common value classes gives a lower bound of the number of geometric intersections, see \cite[Theorem 4.10]{guying}. But, there are something different in the case of self-intersection. Next two Lemmas give some special properties of self-common value classes.

Recall from \cite{guying} that a self-common value class of $\varphi$ is said to be {\em symmetric} if it contains both of $(x',x'')$ and $(x'',x')$.

\begin{Lemma}\label{Lem-symmetric}
Let $\varphi$ be a piecewise-geodesic loop determined by a cyclically $D$-reduced word. Then each symmetric self-common value class of $\varphi$ is not essential, i.e. has index zero.
\end{Lemma}

\begin{proof}
Let $\mu=u_1\cdots u_k$ be a cyclically $D$-reduced word determining $\varphi$. If $\mu$ is trivial, i.e. $m=0$, then there is not any essential self-common value class. Thus, our conclusion is obvious.

Now, we consider the general case: $\mu$ is non-trivial. Let $C$ be a symmetric self-common value class of $\varphi$. Since $\varphi$ is piecewise-geodesic, by Theorem~\ref{th:cvp-component}, we assume that $C$ contains $(e^{\frac{2\pi k_0i}{m}}, e^{\frac{2\pi l_0i}{m}})$ and $(e^{\frac{2\pi l_0i}{m}}, e^{\frac{2\pi k_0i}{m}})$ for integers $k_0, l_0$.

If $k_0 = l_0$, then class $C$ is obvious the whole diagonal $\{(e^{\theta i}, e^{\theta i})\}$ of $S^1\times S^1$, and hence is of type (4) in Theorem~\ref{th:cvp-component}. From Theorem~\ref{thm-index}, we know that $C$ has index $0$, and therefore is an inessential class.

If $k_0\ne l_0$, we may assume that $0\le k_0<l_0<m$. Since $C$ is symmetric, by Theorem~\ref{thm-connect-class}, two pair $(e^{\frac{2\pi k_0i}{m}}, e^{\frac{2\pi l_0i}{m}})$ and $(e^{\frac{2\pi l_0i}{m}}, e^{\frac{2\pi k_0i}{m}})$ lies in the same component of $\cvp(\varphi, \varphi)$. The class $C$ has five possibilities: type (1)-(5), which are listed in Theorem~\ref{th:cvp-component}. Note that the two self-common value pairs mentioned above are distinct. Type (1) is impossible. Since the components of type (4) and (5) have index zero (see Theorem~\ref{thm-index}),  it is sufficient to show that type (2) and (3) are both impossible.

Suppose on the contrary that the class $C$ is a component of self-common value set of type (2). By Theorem~\ref{th:cvp-component} and its Corollary, $C =\{( e^{\frac{2\pi (k+\lambda) i}{m}}, e^{\frac{2\pi (l+\lambda) i}{m}})\mid 0 \le \lambda \le q\}$ for some integers $k,l$ and positive integer $q$. Moreover, $u_{k+r}= u_{l+r}$ for $r=1, 2,\ldots, q$. Thus, $k_0=k+r'$ and $l_0=l+r'$ for some $r'$ with $0\le r'\le q$ because $(e^{\frac{2\pi k_0i}{m}}, e^{\frac{2\pi l_0i}{m}})\in C$. Moreover, $l_0=k+r''$ and $k_0=l+r''$ for some $r''$ with $0\le r''\le q$ because $(e^{\frac{2\pi l_0i}{m}}, e^{\frac{2\pi k_0i}{m}})\in C$. It follows that $2(r''-r') \equiv (l_0-k_0)+(k_0-l_0)\equiv 0 \mod m$. Since $e^{\frac{2\pi k_0i}{m}}$ and $e^{\frac{2\pi l_0i}{m}}$ are distinct and since $0\le k_0<l_0<m$, we have that $m$ is even and $r''-r'=\frac{m}{2}$. It follows that $l_0 = k_0 + \frac{m}{2}$ and hance $q\ge \frac{m}{2}$. We obtain that $u_{k_0+r} = u_{k_0 + \frac{m}{2}+r}$ for $r=0, 1, \ldots, \frac{m}{2}-1$. By the corollary of Theorem~\ref{th:cvp-component}, $C$ would be of type (4). A contradiction.

Suppose on the contrary that the class $C$ is a component of self-common value set of type (3). By Theorem~\ref{th:cvp-component} and its Corollary, $C =\{( e^{\frac{2\pi (k+\lambda) i}{m}}, e^{\frac{2\pi (l+q-\lambda) i}{m}})\mid 0\le \lambda \le q\}$ for some integers $k,l$ and positive integer $q$. Moreover, $u_{k+r}= u_{l+q+1-r}^{-1}$ for $r=1, \ldots, q$, i.e.
$$u_{k+1} = u_{l+q}^{-1}, \ldots, u_{k_0}=u_{l_0}^{-1}, u_{k_0+1}=u_{l_0-1}^{-1},\ldots
 u_{l_0}=u_{k_0}^{-1}, \dots, u_{k+q} = u_{l+1}^{-1}. $$
If $k_0+l_0$ is even, we would obtain that $u_{\frac{k_0+l_0}{2}} = u_{\frac{k_0+l_0}{2}}^{-1}$. This is impossible. If $k_0+l_0$ is odd, we would obtain that $u_{\frac{k_0+l_0-1}{2}} = u_{\frac{k_0+l_0+1}{2}}^{-1}$, which contradicts to the fact that $\mu$ is $D$-reduced.
\end{proof}

By this Lemma, if $(x', x'')$ lies in an essential self-common value class, then $(x'', x')$ must lie in distinct essential class. Thus, the number of essential self-common value classes of any loop on $F_g$ is even. By \cite[Theorem 5.6]{guying}, the half of this number is a lower bound of minimal geometric self-intersection number. (Note that $(x', x'')$ and $(x'', x')$ give the same intersection.) Next Lemma shows that there is still more self-intersection lying in inessential self-common value classes if corresponding element in $\pi_1(F_g, y_0)$ is not prime.

\begin{Lemma}\label{Lem-inessential-non-empty}
Let $\varphi$ be a loop determined by a non-trivial element $\mu^q\in  \pi_1(F_g, y_0)$ with $q>1$. Then  for any $s$, the self-common value class $C_s$ of $\varphi$ determined by $\mu^s$ has zero index, but for any loop $\xi$ homotopic to $\varphi$, the self-common value classes of $\xi$ homotopy determined by $\mu^s$, $s=1,2,\ldots, q-1$ contains at least $q-1$ self-intersections.
\end{Lemma}

\begin{proof}
Without loss of generality, we may assume that $\varphi$ is a piecewise-geodesic loop determined by $\mu^q$, and that $\mu$ is cyclically $D$-reduced. Then the index $0$ is proved in Theorem~\ref{thm-index}.

Consider the standard lifting  $\tilde\varphi_S$ of $\varphi$. Let $\mu=u_1\cdots u_m$, and by homotopy invariance we may assume that $\varphi$ is piecewise-geodesic. By definition of standard liftings (see Definition~\ref{def-standard-lifting}), we have that $\tilde \varphi_S(\frac{s}{q})=\tilde \varphi_S(\frac{ms}{mq}) = \mu^s(\tilde y_0)$. Thus, the set $\cvp(\tilde\varphi_S, \mu^s\tilde\varphi_S)$ of common value pairs contains  a subset $\{(\frac{k+ms}{mq}, \frac{k}{mq})\mid k\in \mathbb{Z}\}$.
Since $\varphi$ is piecewise-geodesic, we obtain that $$\cvp(\tilde\varphi_S, \mu^s\tilde\varphi_S)=\{( \lambda+\frac{s}{q}, \lambda )\mid \lambda\in(-\infty, +\infty) \}.$$
Corresponding self-common value class is
$$\{( e^{2\pi(\lambda+\frac{s}{q})i}, e^{2\pi\lambda i} )\mid \lambda\in(-\infty, +\infty) \}\subset S^1\times S^1.$$ Thus, there are actually $q$ classes, which are determined by $1=\mu^0, \mu^1, \mu^2, \ldots, \mu^{q-1}$. Clearly, the trivial element $1=mu^0$ determines the trivial class consists of the diagonal of $S^1\times S^1$. Other classes are not symmetric except for the class determined by $\mu^{\frac{q}{2}}$ when $q$ is even.

Consider $s$ with $1\le s <q$. Let $\tau$ be an arbitrary loop homotopic to $\varphi$, and $\tilde \tau: \mathbb{R}\to \mathbb{D}^2$ be a lifting homotopic related to the standard lifting $\tilde\varphi_S$  of $\varphi$. Thus, $\tilde\varphi_S$ and $\tilde \tau$ have the same ending points on the circle of infinity of $\mathbb{D}^2$.

We are going to show $p_{S^1}\times p_{S^1}(\cvp(\tilde \tau, \mu^s\tilde \tau))$ contains at least two points. Observe that  $\mu^s\tilde \tau(\lambda) = \tilde \tau(\lambda+\frac{s}{q})$ for all $\lambda\in \mathbb{R}$. If we regard the unique geodesic connecting ending points of $\tilde \tau$ on the circle of infinity  as ``$x$-axis'', the images $\tilde \tau(\mathbb{R})$ and $\mu^s\tilde \tau(\mathbb{R})$ are two ``periodic'' arcs which differ by a translation along the ``$x$-axis''.
\begin{center}
\setlength{\unitlength}{0.68mm}
\begin{picture}(100,24)(0,-12)
\qbezier(-10.,-8.41)(-8.66,-7.69)(-7.22,-6.61)
\qbezier(-7.22,-6.61)(-6.02,-5.71)(-4.66,-4.5)
\qbezier(-4.66,-4.5)(-3.59,-3.54)(-2.24,-2.23)
\qbezier(-2.24,-2.23)(-1.5,-1.5)(0.11,0.11)
\qbezier(0.11,0.11)(1.65,1.65)(2.47,2.44)
\qbezier(2.47,2.44)(3.81,3.74)(4.9,4.7)
\qbezier(4.9,4.7)(6.26,5.91)(7.47,6.79)
\qbezier(7.47,6.79)(8.93,7.86)(10.28,8.56)
\qbezier(10.28,8.56)(11.86,9.38)(13.38,9.73)
\qbezier(13.38,9.73)(15.04,10.11)(16.69,9.95)
\qbezier(16.69,9.95)(18.27,9.8)(19.9,9.13)
\qbezier(19.9,9.13)(21.32,8.56)(22.83,7.57)
\qbezier(22.83,7.57)(24.09,6.74)(25.49,5.59)
\qbezier(25.49,5.59)(26.63,4.64)(27.97,3.38)
\qbezier(27.97,3.38)(28.95,2.46)(30.34,1.07)
\qbezier(30.34,1.07)(36.04,-4.6)(32.69,-1.27)
\qbezier(32.69,-1.27)(34.07,-2.64)(35.08,-3.58)
\qbezier(35.08,-3.58)(36.42,-4.83)(37.57,-5.77)
\qbezier(37.57,-5.77)(38.97,-6.92)(40.24,-7.73)
\qbezier(40.24,-7.73)(41.76,-8.69)(43.2,-9.24)
\qbezier(43.2,-9.24)(44.83,-9.86)(46.43,-9.98)
\qbezier(46.43,-9.98)(48.07,-10.09)(49.73,-9.66)
\qbezier(49.73,-9.66)(51.23,-9.27)(52.81,-8.43)
\qbezier(52.81,-8.43)(54.14,-7.71)(55.59,-6.62)
\qbezier(55.59,-6.62)(56.79,-5.73)(58.15,-4.51)
\qbezier(58.15,-4.51)(59.22,-3.55)(60.57,-2.24)
\qbezier(60.57,-2.24)(61.32,-1.51)(62.92,0.09)
\qbezier(62.92,0.09)(64.47,1.63)(65.28,2.42)
\qbezier(65.28,2.42)(66.62,3.73)(67.71,4.69)
\qbezier(67.71,4.69)(69.07,5.89)(70.28,6.78)
\qbezier(70.28,6.78)(71.74,7.85)(73.09,8.55)
\qbezier(73.09,8.55)(74.67,9.37)(76.19,9.72)
\qbezier(76.19,9.72)(77.85,10.11)(79.5,9.95)
\qbezier(79.5,9.95)(81.08,9.8)(82.71,9.14)
\qbezier(82.71,9.14)(84.13,8.57)(85.64,7.58)
\qbezier(85.64,7.58)(86.9,6.76)(88.3,5.6)
\qbezier(88.3,5.6)(89.44,4.66)(90.78,3.4)
\qbezier(90.78,3.4)(91.76,2.47)(93.16,1.09)
\qbezier(93.16,1.09)(99.89,-5.6)(95.5,-1.25)
\qbezier(95.5,-1.25)(96.89,-2.62)(97.89,-3.56)
\qbezier(97.89,-3.56)(99.23,-4.81)(100.38,-5.75)
\qbezier(100.38,-5.75)(101.78,-6.9)(103.05,-7.71)
\qbezier(103.05,-7.71)(104.57,-8.68)(106.,-9.23)
\qbezier(106.,-9.23)(107.64,-9.86)(109.23,-9.97)
\qbezier(109.23,-9.97)(110.88,-10.09)(112.54,-9.67)
\qbezier(112.54,-9.67)(114.04,-9.28)(115.62,-8.44)
\qbezier(115.62,-8.44)(116.95,-7.72)(118.4,-6.64)
\qbezier(118.4,-6.64)(119.17,-6.06)(120.,-5.37)
\qbezier(-18.,-8.41)(-16.66,-7.69)(-15.22,-6.61)
\qbezier(-15.22,-6.61)(-14.02,-5.71)(-12.66,-4.5)
\qbezier(-12.66,-4.5)(-11.59,-3.54)(-10.24,-2.23)
\qbezier(-10.24,-2.23)(-9.5,-1.5)(-7.89,0.11)
\qbezier(-7.89,0.11)(-6.35,1.65)(-5.53,2.44)
\qbezier(-5.53,2.44)(-4.19,3.74)(-3.1,4.7)
\qbezier(-3.1,4.7)(-1.74,5.91)(-0.53,6.79)
\qbezier(-0.53,6.79)(0.93,7.86)(2.28,8.56)
\qbezier(2.28,8.56)(3.86,9.38)(5.38,9.73)
\qbezier(5.38,9.73)(7.04,10.11)(8.69,9.95)
\qbezier(8.69,9.95)(10.27,9.8)(11.9,9.13)
\qbezier(11.9,9.13)(13.32,8.56)(14.83,7.57)
\qbezier(14.83,7.57)(16.09,6.74)(17.49,5.59)
\qbezier(17.49,5.59)(18.63,4.64)(19.97,3.38)
\qbezier(19.97,3.38)(20.95,2.46)(22.34,1.07)
\qbezier(22.34,1.07)(28.04,-4.6)(24.69,-1.27)
\qbezier(24.69,-1.27)(26.07,-2.64)(27.08,-3.58)
\qbezier(27.08,-3.58)(28.42,-4.83)(29.57,-5.77)
\qbezier(29.57,-5.77)(30.97,-6.92)(32.24,-7.73)
\qbezier(32.24,-7.73)(33.76,-8.69)(35.2,-9.24)
\qbezier(35.2,-9.24)(36.83,-9.86)(38.43,-9.98)
\qbezier(38.43,-9.98)(40.07,-10.09)(41.73,-9.66)
\qbezier(41.73,-9.66)(43.23,-9.27)(44.81,-8.43)
\qbezier(44.81,-8.43)(46.14,-7.71)(47.59,-6.62)
\qbezier(47.59,-6.62)(48.79,-5.73)(50.15,-4.51)
\qbezier(50.15,-4.51)(51.22,-3.55)(52.57,-2.24)
\qbezier(52.57,-2.24)(53.32,-1.51)(54.92,0.09)
\qbezier(54.92,0.09)(56.47,1.63)(57.28,2.42)
\qbezier(57.28,2.42)(58.62,3.73)(59.71,4.69)
\qbezier(59.71,4.69)(61.07,5.89)(62.28,6.78)
\qbezier(62.28,6.78)(63.74,7.85)(65.09,8.55)
\qbezier(65.09,8.55)(66.67,9.37)(68.19,9.72)
\qbezier(68.19,9.72)(69.85,10.11)(71.5,9.95)
\qbezier(71.5,9.95)(73.08,9.8)(74.71,9.14)
\qbezier(74.71,9.14)(76.13,8.57)(77.64,7.58)
\qbezier(77.64,7.58)(78.9,6.76)(80.3,5.6)
\qbezier(80.3,5.6)(81.44,4.66)(82.78,3.4)
\qbezier(82.78,3.4)(83.76,2.47)(85.16,1.09)
\qbezier(85.16,1.09)(91.89,-5.6)(87.5,-1.25)
\qbezier(87.5,-1.25)(88.89,-2.62)(89.89,-3.56)
\qbezier(89.89,-3.56)(91.23,-4.81)(92.38,-5.75)
\qbezier(92.38,-5.75)(93.78,-6.9)(95.05,-7.71)
\qbezier(95.05,-7.71)(96.57,-8.68)(98.,-9.23)
\qbezier(98.,-9.23)(99.64,-9.86)(101.23,-9.97)
\qbezier(101.23,-9.97)(102.88,-10.09)(104.54,-9.67)
\qbezier(104.54,-9.67)(106.04,-9.28)(107.62,-8.44)
\qbezier(107.62,-8.44)(108.95,-7.72)(110.4,-6.64)
\qbezier(110.4,-6.64)(111.17,-6.06)(112.,-5.37)
\put(11.708, 9){\makebox(0,0)[cc]{$\bullet$}}
\put(43, -9.2){\makebox(0,0)[cc]{$\bullet$}}
\put(102, 0){\makebox(0,0)[cc]{$x$}}
\put(-20,0){\vector(1,0){120}}
\end{picture}
\end{center}
By a simple argument of intermediate value theorem, one can prove that  $\mu^s\tilde \tau$ and $\tilde \tau$ must have at least two intersections in each period $1$.

Note that the pair $(x',x'')\in p_{S^1}\times p_{S^1}(\cvp(\tilde\tau, \mu^s\tilde\tau))$ if and only if
the pair $(x'',x')\in p_{S^1}\times p_{S^1}(\cvp(\tilde\tau, \mu^{q-s}\tilde\tau))$. Hence, two classes $p_{S^1}\times p_{S^1}(\cvp(\tilde\tau, \mu^s\tilde\tau))$ and $p_{S^1}\times p_{S^1}(\cvp(\tilde\tau, \mu^{q-s}\tilde\tau))$ give at least two self-intersections for $s=1, 2, \ldots, q-1$ and $s\ne \frac{q}{2}$.
If $q$ is even, the class $p_{S^1}\times p_{S^1}(\cvp(\tilde\tau, \mu^{\frac{q}{2}}\tilde\tau))$ determined by $\mu^{\frac{q}{2}}$ gives at least one self-intersection. There are at least $q-1$ self-intersections in these self-common value classes determined by $\mu, \mu^2, \ldots, \mu^{q-1}$.
\end{proof}

Notice that the self-common value classes of $\varphi$ determined by $\mu^s$ and $\mu^{s+q}$ are the same for all $s$.  In the special case of this Lemma that $s=0$, same as $s=tq$ for all $t$, corresponding common value class consists of the whole diagonal. This class never vanishes, but has nothing to do with the real self-intersections.

\section{Minimum theorems}

In this section, we shall explain that the geometric intersection and self-intersection numbers of loops on surfaces can be derived from the number of essential common value classes. Moreover, the loops realizing their minimal number can be obtained by arbitrary small perturbations on geodesic loops.

\begin{Lemma}\label{lem-prime-push}
Let $\varphi$  be a loop determined by a cyclically $D$-reduced word $\mu=u_1\cdots u_m$. If $\mu$ is prime, then the unique geodesic loop determined by $\mu$ has the minimal self-intersection, which is exactly the half of the number of essential self-common value classes of $\varphi$.
\end{Lemma}

\def\ordc{\succ_c}

\begin{proof}
Let $\varphi_G$ be the unique geodesic loop determined by $\mu$. Since $\mu$ is prime, by Theorem~\ref{th:cvp-component} and Theorem~\ref{thm-connect-class}, there is no self-common value class of type (4) or (5). It follows that any two liftings of $\varphi$ have no common ends on the circle of infinity of $\mathbb{D}^2$. It is known that homotopy related liftings of $\varphi$ and $\varphi_G$ share the same ends. Since $\varphi_G$ is geodesic, any two distinct liftings of $\varphi_G$, as two geodesic lines in $\mathbb{D}^2$, contain at most $1$ common value pair. Moreover, two distinct liftings of $\varphi_G$ contains a common value pair if and only if they determines an essential self-common class. We write $N$ for the number of essential self-common classes of $\varphi$, which is a also that of $\varphi_G$ from the homotopy invariance. From Lemma~\ref{Lem-symmetric}, all essential self-common value classes are non-symmetric. Thus, the number of self-intersections of $\varphi_G$ is $\frac{N}{2}$.

Recall from \cite[Theorem 5.6]{guying}) that $\frac{N}{2}$ is a lower bound of the number self-intersections of all loops in the free homotopy class of $\varphi$. Thus, $\frac{N}{2}$ is the minimal number $\mysi(\varphi)$ of geometric self-intersections.
\end{proof}

Consider the geometric self-intersections of general loops.

\begin{Theorem}\label{thm-minimal-self}
Let $\varphi$  be a non-trivial loop on $F_g$ determined by $\mu^q$ with $q>0$, where $\mu$ is a prime and cyclically $D$-reduced word. Then the geometric self-intersection number $\mysi(\varphi)$ is $\frac{N}{2} +q-1$, where $N$ is the number of essential self-common value classes of $\varphi$. Moreover, a loop realizing its minimal self-intersection can be obtained by an arbitrary small homotopy from the geodesic loop in the loop class of $\varphi$.
\end{Theorem}

\begin{proof}
Note that each essential self-common value class must contain at least one common value pair. By Lemma~\ref{Lem-symmetric}, any essential class is not symmetric, and therefore two essential self-common value classes contribute one self-intersection (see \cite[Theorem 5.6]{guying}). Two pairs $(x', x'')$ and
$(x'', x')$ in $S^1\times S^1$ give the same self-intersection.

Write $\mu=u_1\cdots u_m$. By homotopy invariance we may assume that $\varphi$ is piecewise-geodesic. Let us consider the inessential self-common value classes determined by $\mu^s$, i.e. $\cvp(\tilde\varphi_S, \mu^s\tilde\varphi_S)$. As in the proof of Lemma~\ref{Lem-inessential-non-empty}, we have that
$$(p_{S^1}\times p_{S^1}) (\cvp(\tilde\varphi_S, \mu^s\tilde\varphi_S))
= \{( e^{2\pi(\lambda+\frac{s}{q})i}, e^{2\pi\lambda i} )
 \mid \lambda\in(-\infty, +\infty) \}\subset S^1\times S^1.$$
Note that there are actually $q$  classes, with $s=0,1,\ldots, q-1$.  Clearly, the class determined by $\mu^0=1$ contributes nothing to the self-intersections. The other $q-1$ classes give at least $q-1$ self-intersections. Thus, $\frac{N}{2} +q-1$ is a lower bound for geometric self-intersection number of loops homotopic to $\varphi$.

Let $\varphi'_G$ be the geodesic loop determined by $\mu$. By Lemma~\ref{lem-prime-push}, the number of self-intersection of $\varphi'_G$ is just the minimal geometric self-intersection number of $\varphi'_G$. By Proposition~\ref{prop-s-time}, this number is just $\frac{N}{2q^2}$. Let $\bar\varphi'_G: S^1\times I \to F_g$ be a natural extending of map $\varphi'_G: S^1\to F_g$ so that the image of $\bar\varphi'_R$ is a tubular neighborhood of image of  $\varphi'_G$, and let $\eta: S^1\to S^1\times I$ be a map given by
$$e^{\theta i}  \mapsto
\left\{
\begin{array}{ll}
(e^{2q \theta i}, \frac{\theta}{\pi}) & \mbox{ if } 0\le \theta <\pi, \\
(e^{0 i}, \frac{2\pi - \theta}{\pi}) & \mbox{ if } \pi\le \theta <2\pi. \\
\end{array}
\right.
$$
Then the loop $\bar\varphi'_G\eta$ is represented by the word $\mu^q$, and hence it is homotopic to $\varphi$. Note that after such composition, a self-intersection of $\bar\varphi'_G$ becomes $q^2$ self-intersection of $\bar\varphi'_G\eta$.
\begin{center}
\setlength{\unitlength}{2mm}
\begin{picture}(30,14)(-12,-7)
\qbezier(12.,0.)(12.,1.55)(10.26,2.82)
\qbezier(10.26,2.82)(9.06,3.69)(7.26,4.25)
\qbezier(7.26,4.25)(5.75,4.72)(4.01,4.9)
\qbezier(4.01,4.9)(2.36,5.07)(0.69,4.96)
\qbezier(0.69,4.96)(-1.03,4.84)(-2.58,4.44)
\qbezier(-2.58,4.44)(-4.37,3.98)(-5.66,3.21)
\qbezier(-5.66,3.21)(-7.46,2.14)(-7.89,0.75)
\qbezier(-7.89,0.75)(-8.39,-0.9)(-6.8,-2.38)
\qbezier(-6.8,-2.38)(-5.73,-3.37)(-3.92,-4.03)
\qbezier(-3.92,-4.03)(-2.45,-4.57)(-0.7,-4.81)
\qbezier(-0.7,-4.81)(0.93,-5.04)(2.62,-4.99)
\qbezier(2.62,-4.99)(4.33,-4.94)(5.91,-4.6)
\qbezier(5.91,-4.6)(7.69,-4.22)(9.05,-3.54)
\qbezier(9.05,-3.54)(10.87,-2.64)(11.59,-1.41)
\qbezier(11.59,-1.41)(12.,-0.72)(12.,0.)
\end{picture}
\ \
\begin{picture}(30,14)(-12,-7)
\qbezier(12.,0.)(12.03,1.55)(10.31,2.85)
\qbezier(10.31,2.85)(9.13,3.75)(7.34,4.33)
\qbezier(7.34,4.33)(5.84,4.82)(4.1,5.03)
\qbezier(4.1,5.03)(2.45,5.22)(0.78,5.13)
\qbezier(0.78,5.13)(-0.94,5.04)(-2.5,4.67)
\qbezier(-2.5,4.67)(-4.28,4.24)(-5.61,3.51)
\qbezier(-5.61,3.51)(-7.38,2.53)(-8.01,1.22)
\qbezier(-8.01,1.22)(-8.82,-0.45)(-7.55,-2.06)
\qbezier(-7.55,-2.06)(-6.66,-3.21)(-4.89,-4.04)
\qbezier(-4.89,-4.04)(-3.49,-4.7)(-1.74,-5.07)
\qbezier(-1.74,-5.07)(-0.15,-5.41)(1.55,-5.46)
\qbezier(1.55,-5.46)(3.24,-5.52)(4.87,-5.29)
\qbezier(4.87,-5.29)(6.61,-5.05)(8.09,-4.51)
\qbezier(8.09,-4.51)(9.86,-3.88)(11.,-2.92)
\qbezier(11.,-2.92)(12.58,-1.61)(12.63,-0.03)
\qbezier(12.63,-0.03)(12.67,1.55)(11.15,2.93)
\qbezier(11.15,2.93)(10.04,3.93)(8.3,4.63)
\qbezier(8.3,4.63)(6.84,5.22)(5.1,5.51)
\qbezier(5.1,5.51)(3.49,5.79)(1.8,5.79)
\qbezier(1.8,5.79)(0.1,5.78)(-1.51,5.5)
\qbezier(-1.51,5.5)(-3.25,5.2)(-4.7,4.6)
\qbezier(-4.7,4.6)(-6.44,3.89)(-7.52,2.86)
\qbezier(-7.52,2.86)(-8.98,1.47)(-8.94,-0.14)
\qbezier(-8.94,-0.14)(-8.9,-1.72)(-7.43,-3.09)
\qbezier(-7.43,-3.09)(-6.33,-4.11)(-4.61,-4.83)
\qbezier(-4.61,-4.83)(-3.15,-5.45)(-1.43,-5.77)
\qbezier(-1.43,-5.77)(0.18,-6.07)(1.88,-6.1)
\qbezier(1.88,-6.1)(3.56,-6.12)(5.18,-5.87)
\qbezier(5.18,-5.87)(6.91,-5.6)(8.4,-5.05)
\qbezier(8.4,-5.05)(10.12,-4.4)(11.3,-3.45)
\qbezier(11.3,-3.45)(12.84,-2.21)(13.17,-0.71)
\qbezier(13.17,-0.71)(13.54,0.95)(12.36,2.51)
\qbezier(12.36,2.51)(11.46,3.7)(9.79,4.6)
\qbezier(9.79,4.6)(8.43,5.35)(6.7,5.8)
\qbezier(6.7,5.8)(5.14,6.21)(3.43,6.35)
\qbezier(3.43,6.35)(1.77,6.48)(0.11,6.34)
\qbezier(0.11,6.34)(-1.6,6.2)(-3.15,5.77)
\qbezier(-3.15,5.77)(-4.88,5.3)(-6.23,4.53)
\qbezier(-6.23,4.53)(-7.88,3.6)(-8.75,2.37)
\qbezier(-8.75,2.37)(-9.85,0.82)(-9.48,-0.86)
\qbezier(-9.48,-0.86)(-9.16,-2.37)(-7.7,-3.66)
\qbezier(-7.7,-3.66)(-6.56,-4.67)(-4.87,-5.39)
\qbezier(-4.87,-5.39)(-3.41,-6.01)(-1.7,-6.35)
\qbezier(-1.7,-6.35)(-0.09,-6.67)(1.6,-6.72)
\qbezier(1.6,-6.72)(3.28,-6.77)(4.91,-6.54)
\qbezier(4.91,-6.54)(6.63,-6.31)(8.14,-5.79)
\qbezier(8.14,-5.79)(9.84,-5.21)(11.12,-4.33)
\qbezier(11.12,-4.33)(12.69,-3.26)(13.39,-1.91)
\qbezier(13.39,-1.91)(14.23,-0.28)(13.66,1.39)
\qbezier(13.66,1.39)(13.17,2.85)(11.69,4.06)
\qbezier(11.69,4.06)(10.51,5.04)(8.83,5.74)
\qbezier(8.83,5.74)(7.36,6.35)(5.65,6.68)
\qbezier(5.65,6.68)(4.04,6.99)(2.35,7.04)
\qbezier(2.35,7.04)(0.67,7.08)(-0.96,6.85)
\qbezier(-0.96,6.85)(-2.67,6.61)(-4.19,6.1)
\qbezier(-4.19,6.1)(-5.88,5.52)(-7.18,4.66)
\qbezier(-7.18,4.66)(-8.74,3.62)(-9.52,2.31)
\qbezier(-9.52,2.31)(-10.46,0.72)(-10.1,-0.96)
\qbezier(-10.1,-0.96)(-9.78,-2.49)(-8.42,-3.82)
\qbezier(-8.42,-3.82)(-7.33,-4.89)(-5.69,-5.69)
\qbezier(-5.69,-5.69)(-4.26,-6.38)(-2.56,-6.8)
\qbezier(-2.56,-6.8)(-0.98,-7.19)(0.72,-7.31)
\qbezier(0.72,-7.31)(2.38,-7.42)(4.03,-7.27)
\qbezier(4.03,-7.27)(5.74,-7.12)(7.3,-6.69)
\qbezier(7.3,-6.69)(9.,-6.23)(10.39,-5.48)
\qbezier(10.39,-5.48)(12.,-4.62)(13.03,-3.48)
\qbezier(13.03,-3.48)(14.28,-2.08)(14.48,-0.49)
\qbezier(14.48,-0.49)(14.51,-0.25)(14.51,0.)
\qbezier(12,0)(13.0, -0.2)(14.51,0.)
\end{picture}
\end{center}
Together with $q-1$ intersection of $\eta$, the number of self-intersections of $\bar\varphi'_G\eta$  is $\frac{N}{2} +q-1$.
\end{proof}

This result coincides with the statement of S. P. Tan for surfaces with boundaries, see \cite[Sec. 3, Rem. (2)]{Tan1996}.

It should be mentioned that by using our \gr basis in Theorem~\ref{grobner}, one can tell if an element in $\pi_1(F_g, y_0)$ is prime or not. Especially, we have

\begin{Proposition}
Let $\alpha$ be an element in $\pi_1(F_g, y_0)$. Then $\alpha$ is prime in $\pi_1(F_g, y_0)$ if and only if the cyclic $D$-reduced forms of $\alpha$ is prime.
\end{Proposition}

Now we turn into the geometric intersections of two loops.

\begin{Theorem}\label{thm-minimal-intersection}
Let $\varphi$ and $\psi$ be two loops on oriented surface $F_g$. Then the minimal intersection number $\myi(\varphi, \psi)$ of $\varphi$ and $\psi$ is the same as the number of essential self-common value classes of $\varphi$ and $\psi$. Moreover, the loops realizing this minimal intersection can be obtained by arbitrary small homotopyies from the geodesic loops in the homotopy classes of determined by $\varphi$ and $\psi$, respectively.
\end{Theorem}

\begin{proof}
The proof of this theorem is trivial if one of $\varphi$ and $\psi$ is homotopic to the constant loop.
Now we assume that $\varphi$ and $\psi$  piecewise-geodesic loops, which are respectively determined $D$-reduced non-trivial words $\mu = u_1\cdots u_m$ and $\nu = v_1\cdots v_n$.

If both of $\mu$ and $\nu$ are prime words, as in the proof of Lemma~\ref{lem-prime-push}, we know that the geodesic loops determined respectively by $\mu$ and $\nu$ will realize minimal intersection.

In general, $\varphi$ and $\psi$ are respectively determined by $\mu^s$ and $\nu^t$, where $\mu$ and $\nu$ are both cyclically $D$-reduced and prime, we can use the embedding band technique in the proof of Theorem~\ref{thm-minimal-self}. We shall obtain loops $\varphi''$ and $\psi''$ with intersection number $stN_R$, where $N_R$ is the number of essential common value classes of loops determined by $\mu$ and $\nu$. By proposition~\ref{prop-s-time}, we are done.
\end{proof}

\section{An Example}

In this final section, we give an example, illustrating our method to determine the geometric intersections.

\begin{Example}
Let $\varphi$ and $\psi$ be two loops in the oriented surface $F_2$ of genus $2$, which are determined by $\mu =c_{4}c_{3}c_{4}c_{1}^{-1}, \nu = c_{4}^{-1}c_{3}c_{4}c_{3}^{-1}\in \pi_1(F_2)$, respectively. Then $\myi(\varphi, \psi)=2$.
\end{Example}

By homotopy invariance, we may assume that $\varphi$ and $\psi$ are piecewise-geodesic loops determined by $\mu$ and $\nu$, respectively.   Let us consider the components of common value pairs of $\varphi$ and $\psi$. Starting with $|\mu|\times |\nu| =16$ common value pairs, $(e^{\frac{2k\pi i}{4}}, e^{\frac{2l\pi i}{4}})=(e^{\frac{k\pi i}{2}}, e^{\frac{l\pi i}{2}})$, $k,l  = 1,2,3,4$.

By reading the letters in $\mu$ and $\nu$, since $u_2=v_2=c_3$, the pair $(e^{\frac{\pi i}{2}}, e^{\frac{\pi i}{2}})$ and $(e^{\frac{2\pi i}{2}}, e^{\frac{2\pi i}{2}})$ lie in the same component of $\cvp(\varphi, \psi)$. Since $u_3=v_3=c_4$, the pair $(e^{\frac{3\pi i}{2}}, e^{\frac{3\pi i}{2}})$ also lies in the component mentioned  above. There is no more extension in this direction, and no more extension in another direction because $u_4\ne v_4$, i.e. $u_0\ne v_0$. Hence we have the first data $(1,1,2)_{\mu, \nu}$, which gives a common value class
$$(p_{S^1}\times p_{S^1})(\cvp((u_1u_2u_3u_4)^{-1}\tilde \varphi_S,\ (v_1v_2v_3v_4)^{-1}\tilde \psi_S))$$ (see Theorem~\ref{th:cvp-component} and its corollary).

The pair $(e^{\frac{\pi i}{2}}, e^{\frac{2\pi i}{2}})$ is an isolated point of $\cvp(\varphi, \psi)$, giving the data $(1,2,0)_{\mu, \nu}$.

Consider the pair $(e^{\frac{3\pi i}{2}}, e^{\frac{\pi i}{2}})$. Since $u_3=v_1^{-1}=c_4$, the pairs $(e^{\frac{3\pi i}{2}}, e^{\frac{\pi i}{2}})$ and $(e^{\frac{2\pi i}{2}}, e^{\frac{2\pi i}{2}})$ lie in the same component of $\cvp(\varphi, \psi)$. There is no more extension, and hence we have the data $(2,4,-1)_{\mu, \nu}$.

Finally, we obtain all the components of $\cvp(\varphi, \psi)$:
$$
\begin{array}{llcllcl}
\mbox{data} & \mbox{pair of subwords} & \mbox{index} &
 \mbox{data} & \mbox{pair of subwords} & \mbox{index} \\
 (1,1,2)_{\mu, \nu} & {c_{4}c_{3}c_{4}c_{1}^{-1}},\ {c_{4}^{-1}c_{3}c_{4}c_{3}^{-1}} & 1 &
 (1,2,0)_{\mu, \nu} & {c_{4}c_{3}}, {c_{3}c_{4}}, & 0\\
 (2,4,-1)_{\mu, \nu} & {c_{3}c_{4}c_{1}^{-1}},\ {c_{3}^{-1}c_{4}^{-1}c_{3}} & 0 &
 (2,4,0)_{\mu, \nu} & {c_{3}c_{4}},\ {c_{3}^{-1}c_{4}^{-1}} & 0\\
 (3,1,0)_{\mu, \nu} & {c_{4}c_{1}^{-1}},\ {c_{4}^{-1}c_{3}} & 0 &
 (3,2,0)_{\mu, \nu} & {c_{4}c_{1}^{-1}},\ {c_{3}c_{4}}, & 0\\
 (4,2,1)_{\mu, \nu} & {c_{1}^{-1}c_{4}c_{3}},\ {c_{3}c_{4}c_{3}^{-1}} & 0 &
 (4,3,-2)_{\mu, \nu} & {c_{1}^{-1}c_{4}c_{3}c_{4}},\ {c_{4}c_{3}^{-1}c_{4}^{-1}c_{3}} & -1\\
 (4,3,0)_{\mu, \nu} & {c_{1}^{-1}c_{4}},\ {c_{4}c_{3}^{-1}} & 0 &
 (4,4,0)_{\mu, \nu} & {c_{1}^{-1}c_{4}},\ {c_{3}^{-1}c_{4}^{-1}} & 0\\
\end{array}
$$
The indices of all components can be obtained from Theorem~\ref{thm-index} and (\ref{eq-T}).
Take two components as examples, for the common value class with data $(1,1,2)_{\mu, \nu}$, its index is given by:
$$\begin{array}{cl}
 & \frac{1}{2}(\Theta(-T(u_1), -T(v_1), T(u_2)) + \Theta(T(u_4), T(v_4), -T(u_3)))\\
= & \frac{1}{2}(\Theta(-T(c_4), -T(c_4^{-1}), T(c_3)) + \Theta(T(c_1^{-1}), T(c_3^{-1}), -T(c_4)) )\\
= & \frac{1}{2}(\Theta(e^{\frac{7\pi}{4} i}, e^{\frac{3\pi}{4} i}, e^{\frac{6\pi}{4} i}) + \Theta(e^{\frac{0\pi}{4} i}), e^{\frac{2\pi}{4} i}, e^{\frac{3\pi}{4} i})\\
= & 1.
 \end{array}
$$
For the common value class with data $(1,2,0)_{\mu, \nu}$, its index is given by:
$$\begin{array}{cl}
   & \frac{1}{2}(\Theta(-T(u_1), -T(v_2), T(u_2)) + \Theta(T(u_2), T(v_3), -T(u_1)))\\
 = & \frac{1}{2}(\Theta(-T(c_4), -T(c_3), T(c_3)) + \Theta(T(c_3), T(c_4), -T(c_4)))\\
 = & \frac{1}{2}(\Theta(e^{\frac{7\pi}{4} i}, e^{\frac{2\pi}{4} i},e^{\frac{6\pi}{4} i})
    + \Theta(e^{\frac{6\pi}{4} i}, e^{\frac{3\pi}{4} i}, e^{\frac{7\pi}{4} i}))\\
 = & \frac{1}{2}(1+(-1))\\
 = & 0.
 \end{array}
$$

The two functions $\Theta(\cdot, \cdot, \cdot)$ and $T(\cdot)$ are both symbolic, because they have respectively $3$ and $4g$ possible values. One can compute these indices easily. The numerical computation in \cite[Sec. 6]{Reinhart1962} is avoided.

Next two figures show the intersections in the universal covering (Poincar\'{e} disk) for the first two data: $(1,1,2)_{\mu, \nu}$ and $(1,2,0)_{\mu, \nu}$.

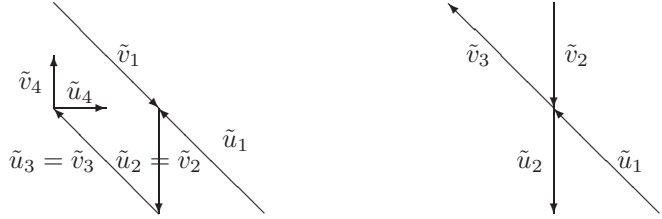
\begin{center}
\setlength{\unitlength}{0.7mm}
\begin{picture}(70,40)(0,-20)
\put(0,20){\vector(1,-1){20}}
\put(40,-20){\vector(-1,1){20}}
\put(20,0){\vector(0,-1){20}}
\put(20,-20){\vector(-1,1){20}}
\put(0,0){\vector(0,1){10}}
\put(0,0){\vector(1,0){10}}

\put(32,-6){\makebox(0,0)[lc]{$\tilde u_1$}}
\put(12,10){\makebox(0,0)[lc]{$\tilde v_1$}}
\put(20,-10){\makebox(0,0)[cc]{$\tilde u_2 = \tilde v_2$}}
\put(8,-10){\makebox(0,0)[rc]{$\tilde u_3 = \tilde v_3$}}
\put(-2,5){\makebox(0,0)[rc]{$\tilde v_4$}}
\put(5,1){\makebox(0,0)[cb]{$\tilde u_4$}}
\end{picture}
\
\
\begin{picture}(40,40)(-20,-20)
\put(20,-20){\vector(-1,1){20}}
\put(0,0){\vector(0,-1){20}}
\put(0,20){\vector(0,-1){20}}
\put(0,0){\vector(-1,1){20}}

\put(12,-10){\makebox(0,0)[lc]{$\tilde u_1$}}
\put(-2,-10){\makebox(0,0)[rc]{$\tilde u_2$}}
\put(2,10){\makebox(0,0)[lc]{$\tilde v_2$}}
\put(-12,10){\makebox(0,0)[rc]{$\tilde v_3$}}
\end{picture}
\end{center}
Here $\tilde u_k$ indicates a lifting of $u_k$, and $\tilde v_l$ indicates a lifting of $v_l$.

Since $\mu$ and $\nu$ are both $D$-cyclically reduced, by Theorem~\ref{thm-connect-class}, each component is exactly a common value class. It follows from Theorem~\ref{thm-minimal-intersection} that $\myi(\varphi, \psi)=2$. Two loops realizing their minimal geometric intersection are the following:
\begin{center}
\setlength{\unitlength}{0.5mm}
\begin{picture}(160,40)(-80,-20)
\put(-72,0){\makebox(0,0)[rc]{$\mu$}}
\put(0,-20){\makebox(0,0)[cc]{$\nu$}}

\qbezier(11.72,-14.14)(13.64,-15.11)(15.81,-15.93)
\qbezier(15.81,-15.93)(17.85,-16.7)(20.06,-17.34)
\qbezier(20.06,-17.34)(22.16,-17.94)(24.4,-18.42)
\qbezier(24.4,-18.42)(26.56,-18.87)(28.8,-19.2)
\qbezier(28.8,-19.2)(30.99,-19.52)(33.25,-19.71)
\qbezier(33.25,-19.71)(35.46,-19.9)(37.71,-19.97)
\qbezier(37.71,-19.97)(39.95,-20.03)(42.18,-19.97)
\qbezier(42.18,-19.97)(44.43,-19.91)(46.65,-19.72)
\qbezier(46.65,-19.72)(48.9,-19.53)(51.09,-19.22)
\qbezier(51.09,-19.22)(53.34,-18.89)(55.5,-18.44)
\qbezier(55.5,-18.44)(57.73,-17.97)(59.84,-17.37)
\qbezier(59.84,-17.37)(62.05,-16.74)(64.08,-15.97)
\qbezier(64.08,-15.97)(66.26,-15.15)(68.19,-14.19)
\qbezier(68.19,-14.19)(70.29,-13.15)(72.06,-11.96)
\qbezier(72.06,-11.96)(74.04,-10.63)(75.55,-9.16)
\qbezier(75.55,-9.16)(77.28,-7.49)(78.36,-5.68)
\qbezier(78.36,-5.68)(79.57,-3.62)(79.89,-1.48)
\qbezier(79.89,-1.48)(80.22,0.76)(79.55,2.98)
\qbezier(79.55,2.98)(78.94,5.02)(77.5,6.96)
\qbezier(77.5,6.96)(76.25,8.64)(74.41,10.19)
\qbezier(74.41,10.19)(72.8,11.56)(70.77,12.78)
\qbezier(70.77,12.78)(68.93,13.88)(66.8,14.85)
\qbezier(66.8,14.85)(64.83,15.74)(62.64,16.49)
\qbezier(62.64,16.49)(60.58,17.2)(58.36,17.77)
\qbezier(58.36,17.77)(56.23,18.32)(53.99,18.74)
\qbezier(53.99,18.74)(51.82,19.14)(49.57,19.42)
\qbezier(49.57,19.42)(47.37,19.69)(45.12,19.84)
\qbezier(45.12,19.84)(42.89,19.98)(40.65,20.)
\qbezier(40.65,20.)(38.41,20.02)(36.18,19.91)
\qbezier(36.18,19.91)(33.93,19.8)(31.72,19.57)
\qbezier(31.72,19.57)(29.47,19.33)(27.29,18.96)
\qbezier(27.29,18.96)(25.04,18.59)(22.9,18.08)
\qbezier(22.9,18.08)(20.67,17.55)(18.59,16.89)
\qbezier(18.59,16.89)(16.39,16.2)(14.39,15.36)
\qbezier(14.39,15.36)(13.,14.78)(11.72,14.14)
\qbezier(18,4)(40,-9)(62,4)
\qbezier(26,0)(40,9)(54,0)

\qbezier(-18,4)(-40,-9)(-62,4)
\qbezier(-26,0)(-40,9)(-54,0)

\qbezier(11.7157,14.1421)(0,8.28427)(-11.7157,14.1421)
\qbezier(11.7157,-14.1421)(0,-8.28427)(-11.7157,-14.1421)

\textcolor[rgb]{1,0,0}{
\qbezier(0,-11.)(1.45,-11.)(2.67,-9.3)
\qbezier(2.67,-9.3)(3.48,-8.17)(4.05,-6.45)
\qbezier(4.05,-6.45)(4.52,-5.04)(4.76,-3.38)
\qbezier(4.76,-3.38)(4.98,-1.84)(5.,-0.22)
\qbezier(5.,-0.22)(5.01,1.39)(4.82,2.94)
\qbezier(4.82,2.94)(4.61,4.6)(4.18,6.03)
\qbezier(4.18,6.03)(3.67,7.74)(2.92,8.93)
\qbezier(2.92,8.93)(1.82,10.67)(0.48,10.95)
\qbezier(0.48,10.95)(0.24,11.)(0,11.)
\qbezier(0,11.)(-1.45,11.)(-2.67,9.3)
\qbezier(-4.05,6.45)(-4.52,5.04)(-4.76,3.38)
\qbezier(-5.,0.22)(-5.01,-1.39)(-4.82,-2.94)
\qbezier(-4.18,-6.03)(-3.67,-7.74)(-2.92,-8.93)
\qbezier(-0.48,-10.95)(-0.24,-11.)(0,-11.)
}
\qbezier(-11.72,14.14)(-13.64,15.11)(-15.81,15.93)
\qbezier(-15.81,15.93)(-17.85,16.7)(-20.06,17.34)
\qbezier(-20.06,17.34)(-22.16,17.94)(-24.4,18.42)
\qbezier(-24.4,18.42)(-26.56,18.87)(-28.8,19.2)
\qbezier(-28.8,19.2)(-30.99,19.52)(-33.25,19.71)
\qbezier(-33.25,19.71)(-35.46,19.9)(-37.71,19.97)
\qbezier(-37.71,19.97)(-39.95,20.03)(-42.18,19.97)
\qbezier(-42.18,19.97)(-44.43,19.91)(-46.65,19.72)
\qbezier(-46.65,19.72)(-48.9,19.53)(-51.09,19.22)
\qbezier(-51.09,19.22)(-53.34,18.89)(-55.5,18.44)
\qbezier(-55.5,18.44)(-57.73,17.97)(-59.84,17.37)
\qbezier(-59.84,17.37)(-62.05,16.74)(-64.08,15.97)
\qbezier(-64.08,15.97)(-66.26,15.15)(-68.19,14.19)
\qbezier(-68.19,14.19)(-70.29,13.15)(-72.06,11.96)
\qbezier(-72.06,11.96)(-74.04,10.63)(-75.55,9.16)
\qbezier(-75.55,9.16)(-77.28,7.49)(-78.36,5.68)
\qbezier(-78.36,5.68)(-79.57,3.62)(-79.89,1.48)
\qbezier(-79.89,1.48)(-80.22,-0.76)(-79.55,-2.98)
\qbezier(-79.55,-2.98)(-78.94,-5.02)(-77.5,-6.96)
\qbezier(-77.5,-6.96)(-76.25,-8.64)(-74.41,-10.19)
\qbezier(-74.41,-10.19)(-72.8,-11.56)(-70.77,-12.78)
\qbezier(-70.77,-12.78)(-68.93,-13.88)(-66.8,-14.85)
\qbezier(-66.8,-14.85)(-64.83,-15.74)(-62.64,-16.49)
\qbezier(-62.64,-16.49)(-60.58,-17.2)(-58.36,-17.77)
\qbezier(-58.36,-17.77)(-56.23,-18.32)(-53.99,-18.74)
\qbezier(-53.99,-18.74)(-51.82,-19.14)(-49.57,-19.42)
\qbezier(-49.57,-19.42)(-47.37,-19.69)(-45.12,-19.84)
\qbezier(-45.12,-19.84)(-42.89,-19.98)(-40.65,-20.)
\qbezier(-40.65,-20.)(-38.41,-20.02)(-36.18,-19.91)
\qbezier(-36.18,-19.91)(-33.93,-19.8)(-31.72,-19.57)
\qbezier(-31.72,-19.57)(-29.47,-19.33)(-27.29,-18.96)
\qbezier(-27.29,-18.96)(-25.04,-18.59)(-22.9,-18.08)
\qbezier(-22.9,-18.08)(-20.67,-17.55)(-18.59,-16.89)
\qbezier(-18.59,-16.89)(-16.39,-16.2)(-14.39,-15.36)
\qbezier(-14.39,-15.36)(-13.,-14.78)(-11.72,-14.14)
%
%
\textcolor[rgb]{0,0,1}{
\qbezier(15.73,-7.05)(17.04,-7.77)(18.59,-8.41)
\qbezier(18.59,-8.41)(19.99,-8.98)(21.57,-9.47)
\qbezier(21.57,-9.47)(23.03,-9.92)(24.62,-10.3)
\qbezier(24.62,-10.3)(26.12,-10.66)(27.71,-10.95)
\qbezier(27.71,-10.95)(29.24,-11.22)(30.84,-11.43)
\qbezier(30.84,-11.43)(32.39,-11.63)(33.99,-11.76)
\qbezier(33.99,-11.76)(35.55,-11.88)(37.14,-11.95)
\qbezier(37.14,-11.95)(38.72,-12.01)(40.3,-12.)
\qbezier(40.3,-12.)(41.89,-11.99)(43.47,-11.92)
\qbezier(43.47,-11.92)(45.06,-11.85)(46.62,-11.7)
\qbezier(46.62,-11.7)(48.22,-11.56)(49.76,-11.35)
\qbezier(49.76,-11.35)(51.36,-11.13)(52.88,-10.84)
\qbezier(52.88,-10.84)(54.48,-10.53)(55.97,-10.16)
\qbezier(55.97,-10.16)(57.56,-9.76)(59.01,-9.28)
\qbezier(59.01,-9.28)(60.59,-8.77)(61.97,-8.17)
\qbezier(61.97,-8.17)(63.52,-7.5)(64.8,-6.75)
\qbezier(64.8,-6.75)(66.29,-5.88)(67.37,-4.91)
\qbezier(67.37,-4.91)(68.7,-3.73)(69.36,-2.46)
\qbezier(69.36,-2.46)(70.17,-0.93)(69.96,0.64)
\qbezier(69.96,0.64)(69.76,2.12)(68.67,3.53)
\qbezier(68.67,3.53)(67.8,4.67)(66.38,5.71)
\qbezier(66.38,5.71)(65.2,6.59)(63.68,7.37)
\qbezier(63.68,7.37)(62.35,8.05)(60.79,8.65)
\qbezier(60.79,8.65)(59.37,9.2)(57.79,9.66)
\qbezier(57.79,9.66)(56.32,10.1)(54.73,10.45)
\qbezier(54.73,10.45)(53.22,10.8)(51.62,11.06)
\qbezier(51.62,11.06)(50.09,11.32)(48.49,11.51)
\qbezier(48.49,11.51)(46.93,11.69)(45.34,11.81)
\qbezier(45.34,11.81)(43.77,11.92)(42.18,11.97)
\qbezier(42.18,11.97)(40.6,12.01)(39.02,11.99)
\qbezier(39.02,11.99)(37.43,11.97)(35.86,11.88)
\qbezier(35.86,11.88)(34.26,11.8)(32.7,11.64)
\qbezier(32.7,11.64)(31.11,11.48)(29.57,11.25)
\qbezier(29.57,11.25)(27.97,11.01)(26.45,10.71)
\qbezier(26.45,10.71)(24.86,10.38)(23.37,9.99)
\qbezier(23.37,9.99)(21.78,9.56)(20.34,9.07)
\qbezier(20.34,9.07)(18.78,8.52)(17.41,7.9)
\qbezier(17.41,7.9)(16.52,7.49)(15.73,7.05)
\qbezier(-15.73,7.05)(-17.04,7.77)(-18.59,8.41)
\qbezier(-18.59,8.41)(-19.99,8.98)(-21.57,9.47)
\qbezier(-21.57,9.47)(-23.03,9.92)(-24.62,10.3)
\qbezier(-24.62,10.3)(-26.12,10.66)(-27.71,10.95)
\qbezier(-27.71,10.95)(-29.24,11.22)(-30.84,11.43)
\qbezier(-30.84,11.43)(-32.39,11.63)(-33.99,11.76)
\qbezier(-33.99,11.76)(-35.55,11.88)(-37.14,11.95)
\qbezier(-37.14,11.95)(-38.72,12.01)(-40.3,12.)
\qbezier(-40.3,12.)(-41.89,11.99)(-43.47,11.92)
\qbezier(-43.47,11.92)(-45.06,11.85)(-46.62,11.7)
\qbezier(-46.62,11.7)(-48.22,11.56)(-49.76,11.35)
\qbezier(-49.76,11.35)(-51.36,11.13)(-52.88,10.84)
\qbezier(-52.88,10.84)(-54.48,10.53)(-55.97,10.16)
\qbezier(-55.97,10.16)(-57.56,9.76)(-59.01,9.28)
\qbezier(-59.01,9.28)(-60.59,8.77)(-61.97,8.17)
\qbezier(-61.97,8.17)(-63.52,7.5)(-64.8,6.75)
\qbezier(-64.8,6.75)(-66.29,5.88)(-67.37,4.91)
\qbezier(-67.37,4.91)(-68.7,3.73)(-69.36,2.46)
\qbezier(-69.36,2.46)(-70.17,0.93)(-69.96,-0.64)
\qbezier(-69.96,-0.64)(-69.76,-2.12)(-68.67,-3.53)
\qbezier(-68.67,-3.53)(-67.8,-4.67)(-66.38,-5.71)
\qbezier(-66.38,-5.71)(-65.2,-6.59)(-63.68,-7.37)
\qbezier(-63.68,-7.37)(-62.35,-8.05)(-60.79,-8.65)
\qbezier(-60.79,-8.65)(-59.37,-9.2)(-57.79,-9.66)
\qbezier(-57.79,-9.66)(-56.32,-10.1)(-54.73,-10.45)
\qbezier(-54.73,-10.45)(-53.22,-10.8)(-51.62,-11.06)
\qbezier(-51.62,-11.06)(-50.09,-11.32)(-48.49,-11.51)
\qbezier(-48.49,-11.51)(-46.93,-11.69)(-45.34,-11.81)
\qbezier(-45.34,-11.81)(-43.77,-11.92)(-42.18,-11.97)
\qbezier(-42.18,-11.97)(-40.6,-12.01)(-39.02,-11.99)
\qbezier(-39.02,-11.99)(-37.43,-11.97)(-35.86,-11.88)
\qbezier(-35.86,-11.88)(-34.26,-11.8)(-32.7,-11.64)
\qbezier(-32.7,-11.64)(-31.11,-11.48)(-29.57,-11.25)
\qbezier(-29.57,-11.25)(-27.97,-11.01)(-26.45,-10.71)
\qbezier(-26.45,-10.71)(-24.86,-10.38)(-23.37,-9.99)
\qbezier(-23.37,-9.99)(-21.78,-9.56)(-20.34,-9.07)
\qbezier(-20.34,-9.07)(-18.78,-8.52)(-17.41,-7.9)
\qbezier(-17.41,-7.9)(-16.52,-7.49)(-15.73,-7.05)
\qbezier(-15.73,-7.05)(0,3)(15.73,-7.05)
\qbezier(-15.73,7.05)(0,0)(15.73,7.05)
}
\end{picture}
\end{center}

Our method to compute intersections and self-intersections still works for the loops on the surfaces with non-empty boundaries, as in \cite{Tan1996}. In this case, a cyclically reduced word is the same as one without any cyclical cancelation. The computation of components is exactly the same. The index of a common value class can be computed if the function $\Theta(\cdot, \cdot, \cdot)$ is known, which is clearly available when the generators in $\pi_1$ are given precisely as concrete loops.



\end{document}